\documentclass[11pt]{article}

\usepackage{fullpage}
\usepackage{parskip}
\usepackage{tikz} 
\usepackage{hyperref}
\usepackage{eqnarray,amsthm, amssymb, amsmath,verbatim,epsfig}
\usepackage{mathrsfs}
\usepackage[margin=1in]{geometry}
\usepackage{enumerate}
\usepackage{graphicx}
\usepackage{epstopdf}
\usepackage{xcolor}

\usepackage{pgfplots}

\usetikzlibrary{calc,shadings,patterns,arrows,decorations.pathreplacing}
\newcolumntype{C}[1]{>{\centering\let\newline\\\arraybackslash\hspace{0pt}}m{#1}}

\newtheorem{theorem}{Theorem}[section]
\newtheorem{lemma}[theorem]{Lemma}
\newtheorem{corollary}[theorem]{Corollary}

\newtheorem{conjecture}{Conjecture}[section]

\newtheorem{problem}{Problem}[section]

\newcommand{\tref}[1]{Theorem \ref{theorem:#1}}
\newcommand{\lref}[1]{Lemma \ref{lemma:#1}}
\newcommand{\fref}[1]{Figure \ref{fig:#1}}

\newcommand{\cref}[1]{Conjecture \ref{conjecture:#1}}
\newcommand{\coref}[1]{Corollary \ref{corollary:#1}}

\def\addlegendimage{\csname pgfplots@addlegendimage\endcsname}
\pgfkeys{/pgfplots/number in legend/.style={/pgfplots/legend image code/.code={\node at (0.15,-0.0225){#1};},},}
\pgfplotsset{every legend to name picture/.style={west}}

\newcommand{\drawtab}[3]{\begin{tikzpicture}[scale =.4] 
	\draw[help lines] (0,0) grid (#1 + #2 + #3 + 1, 1);
	\ifnum #1>0 
	\foreach \dir in {1,...,#1}{\node at (\dir-.5,.5) {\footnotesize$1$};};\fi 
	\ifnum #2>0 
	\foreach \dir in {1,...,#2}{\node at (#1+\dir-.5,.5) {\footnotesize$2$};};\fi 
	\ifnum #3>0 
	\foreach \dir in {1,...,#3}{\node at (#1+#2+\dir-.5,.5) {\footnotesize$\overline{2}$};};\fi 
	\node at (#1+#2+#3+.5,.5) {\footnotesize$\widehat{2}$};
	\end{tikzpicture}}

\newcommand{\fillshade}[1]{\foreach \x/\y in {#1}{\path[fill,blue!20!white] (\x-1,\y-1) rectangle (\x,\y);}}

\newcommand{\fillll}[3]{\node at (#1-.5,#2-.5) {\footnotesize$#3$};}
\newcommand{\fillsome}[1]{\foreach \x/\y/\z in {#1}{\node at (\x-.5,\y-.5) {\footnotesize$\z$};}}

\newcommand{\PFmnu}[4]{
	\coordinate (axis) at (#1);
	\coordinate (axis2) at (0,0);
	\draw[help lines] (#1) grid +(#2,#3);
	\draw[dashed] (#1) -- +(#2,#3);
	\coordinate (prev) at (#1);
	
	\foreach \x/\y in {#4}{	
		\ifnum\y=0
		\draw[line width=2pt,blue] (axis)+(axis2) -- +(#2,0);
		\else
		\draw[line width=2pt,blue] (axis)+(\x,0) -- +(\x,1);
		\draw[line width=2pt,blue] (axis)+(axis2) -- +(\x,0);
		\path (axis)+(\x+0.5,0.5) node {$\y$};		
		\fi
		\path (axis) -- +(0,1) coordinate (axis);
		\path (axis2) -- (\x,0) coordinate (axis2);
	}
}
\newcommand{\PFmnum}[4]{
	\coordinate (axis) at (#1);
	\coordinate (axis2) at (0,0);
	\draw[help lines] (#1) grid +(#2,#3);
	\draw[dashed] (#1) -- +(#2,#3);
	\coordinate (prev) at (#1);
	\foreach \x/\y in {#4}{	
		\ifnum\y=0
		\draw[line width=1pt,blue] (axis)+(axis2) -- +(#2,0);
		\else
		\draw[line width=1pt,blue] (axis)+(\x,0) -- +(\x,1);
		\draw[line width=1pt,blue] (axis)+(axis2) -- +(\x,0);
		\path (axis)+(\x+0.5,0.5) node {\tiny$\y$};		
		\fi
		\path (axis) -- +(0,1) coordinate (axis);
		\path (axis2) -- (\x,0) coordinate (axis2);
	}
}

\newcommand{\Dpath}[4]{
	\coordinate (axis) at (#1);
	\coordinate (axis2) at (#1);
	\draw[help lines] (#1) grid +(#2,#3);
	\draw[dashed] (#1) -- +(#2,#3);
	\foreach \x in {#4}{	
		\ifnum\x=-1
		\draw[line width=2pt,blue] (axis)+(axis2) -- +(#2,0);
		\else
		\draw[line width=2pt,blue] (axis)+(\x,0) -- +(\x,1);
		\draw[line width=2pt,blue] (axis)+(axis2) -- +(\x,0);
		\path (axis) -- +(0,1) coordinate (axis);
		\path (axis2) -- (\x,0) coordinate (axis2);	
		\fi
	}
}

\newcommand{\PFtext}[2]{
	\coordinate (axis) at (#1);
	\coordinate (axis2) at (#1);
	\foreach \x/\y in {#2}{	
		\path (axis)+(\x+0.5,0.5) node {$\y$};		
		\path (axis) -- +(0,1) coordinate (axis);
		\path (axis2) -- (\x,0) coordinate (axis2);
}}

\newcommand{\PFad}[2]{
	\coordinate (axis) at (#1);
	\foreach \x/\y/\z in {#2}{	
		\path (axis)+(0.5,0.5) node {$\x$};
		\path (axis)+(3,0.5) node {$\y$};
		\path (axis)+(5.5,0.5) node {$\z$};	
		\path (axis) -- +(0,1) coordinate (axis);
	}
	\path (axis)+(0.5,0.5) node {$i$};
	\path (axis)+(3,0.5) node {$a_i(\PF)$};
	\path (axis)+(6,0.5) node {$d_i(\PF)$};	
}

\def\Z{\mathbb{Z}}

\newcommand{\la}{\lambda}
\newcommand{\sg}{\sigma}
\newcommand{\qbinom}[2]{\left[#1 \atop #2  \right]_q}

\newcommand{\thn}[1]{$#1\textnormal{th}$}
\newcommand{\PF}{\mathrm{PF}}

\newcommand{\dinv}{\mathrm{dinv}}

\newcommand{\area}{\mathrm{area}}

\newcommand{\qn}[1]{[#1]_{q}}

\newcommand{\Sn}[1]{\mathcal{S}_{#1}}
\newcommand{\Snn}{\mathcal{S}_{n}}

\newcommand{\La}{\Lambda}
\newcommand{\Inj}{\mathrm{Tab}}

\newcommand{\al}{\alpha}
\newcommand{\svDash}{\vDash_{\mathrm{strong}}}
\newcommand{\C}{\mathbb{C}}
\newcommand{\MH}{\widetilde{H}}
\newcommand{\inv}{\mathrm{inv}}
\newcommand{\maj}{\mathrm{maj}}

\newcommand{\Dyck}{\mathcal{D}}

\newcommand{\ga}{\gamma}
\newcommand{\valley}{\mathrm{valley}}

\newcommand{\Rise}{\mathrm{Rise}}
\newcommand{\Val}{\mathrm{Val}}

\newcommand{\WPFc}{\mathcal{WPF}}
\newcommand{\OP}{\mathcal{OP}}
\newcommand{\minimaj}{\mathrm{minimaj}}
\newcommand{\ind}{\mathrm{ind}}
\newcommand{\des}{\mathrm{des}}
\newcommand{\miniword}{\mathrm{miniword}}
\newcommand{\stat}{\mathrm{stat}}
\newcommand{\be}{\beta}
\newcommand{\OPa}{\mathcal{OP}^{\mathrm{all}}}
\newcommand{\Ia}{I^{\mathrm{all}}}
\newcommand{\Ma}{M^{\mathrm{all}}}

\def\multiset#1#2{\ensuremath{\left(\kern-.3em\left(\genfrac{}{}{0pt}{}{#1}{#2}\right)\kern-.3em\right)}}

\title{The valley version of the Extended Delta Conjecture}

\author{
Dun Qiu\\[-0.8ex]
\small Department of Mathematics\\[-0.8ex]
\small University of California, San Diego\\[-0.8ex]
\small La Jolla, CA 92093, USA\\[-0.8ex]
\small \texttt{duqiu@ucsd.edu}
\and 
Andrew Timothy Wilson\textsuperscript{1}\\[-0.8ex]
\small Department of Mathematics \& Statistics\\[-0.8ex]
\small Portland State University\\[-0.8ex]
\small Portland, OR 97201, USA\\[-0.8ex]
\small \texttt{andwils2@pdx.edu}
}
\date{}

\begin{document}
\maketitle
\begin{abstract}
\noindent 
The {Shuffle Theorem} of Carlsson and Mellit gives a combinatorial expression for the bigraded Frobenius characteristic of the ring of diagonal harmonics, and the {Delta Conjecture} of Haglund, Remmel and the second author provides two generalizations of the Shuffle Theorem to the delta operator expression $\Delta'_{e_k} e_n$. 
Haglund et al.\ also propose the Extended Delta Conjecture for the delta operator expression $\Delta'_{e_k} \Delta_{h_r}e_n$, which is analogous to the \emph{rise version} of the Delta Conjecture. Recently, D'Adderio, Iraci and Wyngaerd proved the  rise version of the Extended Delta Conjecture at the case when $t=0$.
In this paper, we propose a new \emph{valley version} of the Extended Delta Conjecture. Then, we work on the combinatorics of extended ordered multiset partitions to prove that the two conjectures for $\Delta'_{e_k} \Delta_{h_r}e_n$ are equivalent when $t$ or $q$ equals 0, thus proving the valley version of the Extended Delta Conjecture when $t$ or $q$ equals 0.
\ \\

\end{abstract}

\footnotetext[1]{The second author was partially supported by an NSF Mathematical Sciences Postdoctoral Research Fellowship.}

\section{Introduction}
Let $X=\{x_1,x_2,\ldots,x_n\}$ and $Y=\{y_1,y_2,\ldots,y_n\}$ be two sets of $n$ commuting variables. The {\em ring of diagonal harmonics} consists of 
those polynomials in $\mathbb{Q}[X,Y]$ which satisfy 
the following system of differential equations
$$
\partial_{x_1}^a\partial_{y_1}^b\,f(X,Y)+\partial_{x_2}^a\partial_{y_2}^b\,f(X,Y)+\ldots+\partial_{x_n}^a\partial_{y_n}^b\,f(X,Y)=0
$$
for each pair of integers $a$ and $b$ such that {$a+b>0$}. Haiman \cite{Haiman} proved that the {\em bigraded Frobenius characteristic} of the $\Sn{n}$-module of diagonal harmonics, $DH_n(X;q,t)$, is given by
\begin{equation}
DH_n(X;q,t)=\nabla e_n.
\end{equation}
where $\nabla$, $e_n$,  and other symmetric function notation will be defined in Section 2.
The {\em Classical Shuffle Conjecture} proposed by Haglund, Haiman, Loehr, Remmel and Ulyanov \cite{HHLRU} gives a well-studied combinatorial expression for the bigraded Frobenius characteristic of the ring of diagonal harmonics. The Shuffle Conjecture has been proved by Carlsson and Mellit \cite{CM} as  the \emph{Shuffle Theorem} as follows; again, relevant notation will be given in Section 2.

\begin{theorem}[Carlsson and Mellit]\label{theorem:shuffle} For any integer $n\geq 0$, 
	\begin{equation}
	\nabla e_n=\sum_{\PF\in\WPFc_n}t^{\area(\PF)}q^{\dinv(\PF)} X^{\PF},
	\end{equation}
\end{theorem} 
\noindent which says that the Frobenius characteristic of diagonal harmonics can be written as a generating function of combinatorial objects called word parking functions. 
As a generalization of the Shuffle Theorem, the Delta Conjecture can be stated as
\begin{conjecture}[Haglund, Remmel and Wilson]
	For any integers $n> k\geq 0$,
	\begin{eqnarray}
	\Delta'_{e_k}e_n&=&\sum_{\PF\in\WPFc_n}t^{\area(\PF)}q^{\dinv(\PF)} X^{\PF}
	\prod_{i\in \Rise(\PF)}
	(1+\frac{z}{t^{a_i(\PF)}})\bigg|_{z^{n-k-1}}\label{delta1}\\
	&=&\sum_{\PF\in\WPFc_n}t^{\area(\PF)}q^{\dinv(\PF)} X^{\PF}
	\prod_{i\in \Val(\PF)}
	(1+\frac{z}{q^{d_i(\PF)+1}})\bigg|_{z^{n-k-1}}.\label{delta2}
	\end{eqnarray}
\end{conjecture}
The Delta Conjecture has two versions, the \emph{rise version} (Equation \ref{delta1}) and the \emph{valley version} (Equation \ref{delta2}), which are different generating functions about parking functions. The Delta Conjecture is still open, but several cases of the Delta Conjecture have been proved. The conjecture for $\Delta_{e_1}e_n$ is proved by Haglund, Remmel and
the second author \cite{HRW};
the rise version Delta Conjecture at $q = 1$ is proved by Romero \cite{Romero}; the ``Catalan'' case of the conjecture is proved by Zabrocki \cite{Zabdelta}. 
The Delta Conjecture at the case when $t$ or $q$ equals $0$ is proved by Garsia, Haglund, Remmel and Yoo \cite{Delta0}; the second author \cite{wilson}; Rhoades \cite{Rhoades}; Haglund, Rhoades and Shimozono \cite{Delta01}. 

In \cite{HRW}, the authors also conjectured a combinatorial formula for the expression $\Delta'_{e_k} \Delta_{h_r}e_n$ which we call the \emph{Extended Delta Conjecture}, and the combinatorial side is a generating function of the set of \emph{extended word parking functions with blank valleys}. The \emph{Extended Delta Conjecture} of Haglund, Remmel and the second author \cite{HRW} is as follows.

\begin{conjecture}[Rise version of the Extended Delta Conjecture \cite{HRW}]\label{conjecture:deltar1}
	For any positive integers $n$, $k$, and $r$ with $k<n$,
	$$
	\Delta'_{e_k}\Delta_{h_r}e_{n}=\sum_{\PF\in\WPFc_{n;r}}t^{\area(\PF)}q^{\dinv(\PF)}x^\PF \prod_{i\in \Rise(\PF)}\left(1+\frac{z}{t^{a_i(\PF)}}\right) \bigg|_{z^{n-k-1}},
	$$
\end{conjecture}
\noindent which is analogous to the rise version Delta Conjecture.

By defining contractible valley set of parking functions with blank valleys, we conjecture the following.
\begin{conjecture}[Valley version of the Extended Delta Conjecture]\label{conjecture:deltar2}
	For any positive integers $n$, $k$, and $r$ with $k<n$,
	$$
	\Delta'_{e_k}\Delta_{h_r}e_{n}=\sum_{\PF\in\WPFc_{n;r}}t^{\area(\PF)}q^{\dinv(\PF)}x^\PF \prod_{i\in \Val(\PF)}\left(1+\frac{z}{q^{d_i(\PF)+1}}\right) \bigg|_{z^{n-k-1}}.
	$$
\end{conjecture}
\noindent We call \cref{deltar2} the \emph{valley version Extended Delta Conjecture}.
Very recently, D'Adderio, Iraci and Wyngaerd  \cite{Michele} proved this conjecture in the case $t=0$. However, the valley version conjecture of $\Delta'_{e_k} \Delta_{h_r}e_n$ is new and has not appeared anywhere before. We believe that the valley version conjecture is true since we have verified the conjecture for $n\leq 10$ by Maple programs, and we have also proved the valley version conjecture at the case when $t$ or $q$ is zero.

The organization of this paper is as follows. In Section 2, we shall introduce some background about symmetric functions and parking functions related to the Delta Conjecture.
In Section 3, we  introduce ordered multiset partitions, extended ordered multiset partitions and their connections to the Delta Conjectures. In Section 4, we  prove that the statistics inv, maj and dinv are equi-distributed by three insertion algorithms. In Section 5, we prove that the statistics inv and minimaj are equi-distributed by generalizing a method of Rhoades \cite{Rhoades}, which completes a proof of the {valley version} conjecture of $\Delta'_{e_k} \Delta_{h_r}e_n$ when $t$ or $q$ equals 0. In Section 6, we give a brief summary of this paper and point out some future directions of this research.

\section{Background}
We shall introduce some algebraic and combinatorial background about symmetric functions and parking functions that is involved in the Delta Conjecture. We shall start with definitions about symmetric functions.

The \emph{symmetric group} $\Snn$ is the set of permutations of size $n$.
Given any permutation $\sg=\sg_1\cdots\sg_n\in\Snn$, the \emph{descent number} of $\sg$ is defined to be $\des(\sg):=|\{i:\sg_i>\sg_{i+1}\}|$, and the \emph{major index} of $\sg$ is
$\maj(\sg):=\sum_{\sg_i>\sg_{i+1}} i$.

For any integer $n$, a weakly decreasing sequence of positive integers $\la=(\la_1,\ldots,\la_k)$ is a \emph{partition} (or an \emph{integer partition}) of $n$ if 
$\sum_{i=1}^{k} \la_i = n$, written $\la\vdash n$. We let  $|\la|=n$ and $\ell(\la)=k$ denote the size and length (number of parts) of the partition $\la$. 

A \emph{weak composition} of an integer $n$ is defined to be a sequence of \emph{non-negative} integers $\al=(\al_1,\ldots,\al_k)$ such that
$\sum_{i=1}^{k} \al_i = n$, written $\al\vDash n$; and a \emph{strong composition} of $n$ is defined to be a sequence of \emph{positive} integers $\al=(\al_1,\ldots,\al_k)$ such that
$\sum_{i=1}^{k} \al_i = n$, written $\al\svDash n$. We let $|\al|=n$ and $\ell(\al)=k$ denote the size and the length of the composition $\al$, respectively.

For each partition $\la=(\la_1,\ldots,\la_k)\vdash n$, we can associate to the partition a \emph{Ferrers diagram} in French notation, which is a diagram with $n$ squares such that there are $\la_i$ squares in the \thn{i} row, counting from bottom to top. 
For each cell $c\in\la$, we let the \emph{coarm}  of $c$, $a_{\la}'(c)$, be the number of cells to the right of $c$; the \emph{coleg}  of $c$, $\ell_{\la}'(c)$, be the number of cells below $c$. We often abbreviate the notations to $a'(c)$ and $\ell'(c)$, and we let $(a_{\la}'(c), l_{\la}'(c))$ denote the coordinate of $c$. \fref{Ferrers} shows an example of the Ferrers diagram of a partition $\la=(7,7,5,3,3)\vdash 25$.

\begin{figure}[ht!]
	\centering
	\vspace{-1mm}	
	\begin{tikzpicture}[scale=0.5]
	\draw[help lines] (0,0) grid (5,3);
	\draw[help lines] (0,0) grid (3,4);
	\draw[help lines] (0,0) grid (7,2);
	\fillll{3}{3}{c};
	\draw[<->,thick] (2.5,0) -- (2.5,2);
	\draw[<->,thick] (0,2.5) -- (2,2.5);
	\node at (3.8,.8) {\footnotesize $\ell_{\la}'(c)$};
	\node at (1,1.8) {\footnotesize $a_{\la}'(c)$};
	\node at (9,2) {\small
		\begin{tabular}{r@{\hskip 1.5mm}l}
		$a_{\la}'(c)= 2$\\
		$\ell_{\la}'(c)= 2$\\
		\end{tabular}};
	\end{tikzpicture}
	\caption{The Ferrers diagram of the partition $\la=(7,7,5,3,3)$.}
	\label{fig:Ferrers}
\end{figure}
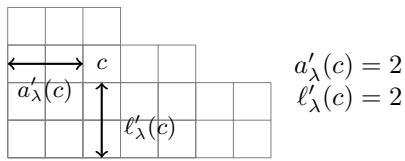

Now let $\la$ be a partition of $n$. We can fill the cells of the Ferrers diagram of $\la$ with positive integers to obtain a \emph{tableau} $T$. The set of tableaux of shape $\la$ is denoted by $\Inj(\la)$.  We also use $T$ to denote the multiset of the filled integers, and we write $X^T:=\prod_{i\in T} x_i$.

Let $\La$ denote the ring of symmetric functions with coefficients in $\C(q,t)$, and let $\La^{(n)}$ denote the elements of $\La$ that are homogeneous of degree $n$. The \emph{elementary symmetric function basis} $\{e_\mu\}_{\mu\vdash n}$ of $\La^{(n)}$ is defined by 
$$
e_k:=\sum_{i_1<\cdots<i_k}x_{i_1}\cdots x_{i_k},\quad\mbox{and}\quad e_\mu:=e_{\mu_1}\cdots e_{\mu_{\ell(\mu)}},
$$
and the \emph{homogeneous symmetric function} basis $\{h_{\mu}\}_{\mu\vdash n}$ is
defined by 
$$
h_k:=\sum_{i_1\leq\cdots\leq i_k}x_{i_1}\cdots x_{i_k},\quad\mbox{and}\quad h_\mu:=h_{\mu_1}\cdots h_{\mu_{\ell(\mu)}}.
$$

Macdonald \cite{Macbook} introduced  a family of orthogonal symmetric functions known as \emph{Macdonald polynomials}, which have nice mathematical and physical properties.
Macdonald polynomials have several transformations, and the form that we are using is called the \emph{modified Macdonald polynomials} $\MH_{\mu}[X;q,t]$ indexed by partitions $\mu\vdash n$. One combinatorial way to define $\MH_\mu[X;q,t]$ is due to the work of Haglund, Haiman and Loehr \cite{HHL}:
$$
\MH_\mu[X;q,t] := \sum_{T\in\Inj(\mu)} q^{\inv(T)}t^{\maj(T)}X^T,
$$
where inv and maj are two statistics defined on the tableau $T$. We shall often abbreviate $\MH_\mu[X;q,t]$ to $\MH_\mu$.

The symmetric function operators nabla ($\nabla$), delta ($\Delta$), and delta prime ($\Delta'$) are eigenoperators of Macdonald polynomials defined by Bergeron and Garsia \cite{GB}.
For any partition $\mu\vdash n$, we let
\begin{equation*}
B_\mu := \sum_{c \in\mu} q^{a'(c)} t^{l'(c)}
\quad\mbox{and}\quad
T_\mu := \prod_{c \in\mu} q^{a'(c)} t^{l'(c)}
\end{equation*}
be polynomials defined from the Ferrers diagram of $\mu$. Given a modified Macdonald polynomial $\MH_{\mu}[X;q,t]$, the operator \emph{nabla} acts by
\begin{equation*}
\nabla \MH_{\mu}:= T_\mu \MH_{\mu}
\end{equation*}
and we extend by scalars to obtain a symmetric function operator.
Let $f$ be a given symmetric function, then $\Delta_f$ and $\Delta'_f$ are the operators such that
\begin{equation*}
\Delta_f \MH_{\mu}:= f[B_\mu] \MH_{\mu}, \quad \Delta'_f \MH_{\mu}:= f[B_\mu-1] \MH_{\mu}, 
\end{equation*}
where $f[B_\mu]$ and $f[B_\mu-1]$ are plethystic expressions which can be thought of as substitutions.

For example, for the partition $\mu=(3,1)\vdash 4$, we can first draw its Ferrers diagram, and fill in each cell $c\in\mu$ the weight $q^{a'(c)} t^{l'(c)}$. This process is pictured in \fref{ptnmu}.
\begin{figure}[ht!]
	\centering
	\begin{tikzpicture}[scale=.6]
	\draw[thick] (3,1) grid (0,0) grid (1,2);
	\fillll{1}{1}{1};\fillll{2}{1}{q};\fillll{3.1}{1.1}{q^2};\fillll{1}{2}{t};
	\end{tikzpicture}
	\caption{A partition $\mu=(3,1)$.}
	\label{fig:ptnmu}
\end{figure}
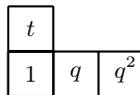

By definition, we have $B_{(3,1)}=1+q+q^2+t$, $T_{(3,1)}=q^3t$, and 
$\nabla \MH_{(3,1)}= q^3t\: \MH_{(3,1)}$.
Setting the symmetric function $f=e_2$, then 
\begin{equation*}
\Delta_{e_2} \MH_{(3,1)} = e_2[1+q+q^2+t]\: \MH_{(3,1)}= (q+q^2+t+q^3+qt+q^2t)\: \MH_{(3,1)},
\end{equation*}
and 
\begin{equation*}
\Delta'_{e_2} \MH_{(3,1)} = e_2[q+q^2+t]\: \MH_{(3,1)}= (q^3+qt+q^2t)\: \MH_{(3,1)}.
\end{equation*}

Note that for $\mu\vdash n$, $e_n[B_\mu] = e_{n-1}[B_\mu -1] = T_\mu$, thus for any $F \in \La^{(n)}$,
$\nabla F = \Delta_{e_n} F = \Delta'_{e_{n-1}} F$. 
Furthermore, since $e_k[X+1]=e_k[X]+e_{k-1}[X]$, we have the following relation between the operators $\Delta$ and $\Delta'$:
\begin{equation}
\Delta_{e_k} = \Delta'_{e_k} +\Delta'_{e_{k-1}}.
\end{equation}

For integer $n,k$, we define the $q$-analogues of $n$, $n!$ and $\binom{n}{k}$ to be
$$
[n]_q:=\frac{1-q^n}{1-q}, \quad [n]_q!: = [1]_q [2]_q \cdots [n]_q, \quad\mbox{and}\quad
\qbinom{n}{k} := \frac{\qn{n}!}{\qn{k}!\qn{n-k}!}.
$$

We also define several combinatorial objects that are related to the Delta Conjecture.

Let $n$ be a positive integer. An $(n,n)$-{\em Dyck path} $P$ is a lattice path from $(0,0)$ to $(n,n)$ which always remains weakly above the main diagonal $y=x$. The number of Dyck paths of size $n$ is given by the \thn{n} Catalan number $C_n=\frac{1}{n+1}\binom{2n}{n}$. We let $\Dyck_n$ denote the set of Dyck paths of size $n$.

For a Dyck path $P\in\Dyck_n$, the cells that are cut through by the main diagonal are called  \emph{diagonal cells}, and the cells between the diagonal cells and the path are called \emph{area cells}. We call the main diagonal the \thn{0} diagonal; we call the line that parallel to and above the main diagonal with distance $i$ the \thn{i} diagonal.

Given an $(n,n)$-Dyck path $P$, an \emph{$(n,n)$-word parking function} (or a \emph{labeled Dyck path}) PF is obtained by labeling the north steps of $P$ with positive integers such that the labels (called \emph{cars}) are strictly increasing along each column of $P$. We let $\ell_i(\PF)$ be the \thn{i} row label of PF. 
We let $\WPFc_n$ denote the set of $(n,n)$-word parking functions. We shall also omit ``word" to call $\PF\in\WPFc_n$ a \emph{parking function} in this paper.

For a parking function $\PF\in\WPFc_n$, let $a_i(\PF)$ be the number of full cells between the path and the diagonal in the \thn{i} row counting from bottom to top, and let
\begin{multline*}
d_i(\PF):=\big|\{(i,j):i<j,\ a_i(\PF)=a_j(\PF)\textnormal{ and }\ell_i(\PF)<\ell_j(\PF)\}\\
\cup\{(i,j):i<j,\ a_i(\PF)=a_j(\PF)+1\textnormal{ and }\ell_i(\PF)>\ell_j(\PF)\}\big|,
\end{multline*}
then $\area(\PF):=\sum_{i=1}^{n}a_i(\PF)$ is the \emph{area} of PF and $\dinv(\PF):=\sum_{i=1}^{n}d_i(\PF)$ is the \emph{dinv} of PF. \fref{areadinv} gives an example of a $(7,7)$-parking function with area 13 and dinv 2.

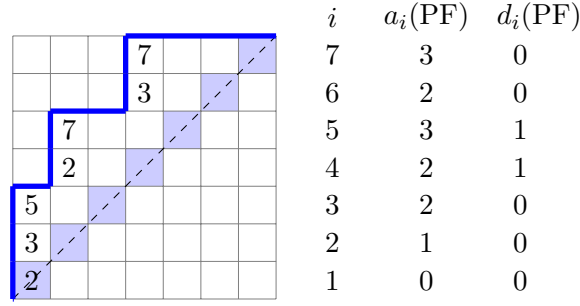
\begin{figure}[ht!]
	\centering
	\begin{tikzpicture}[scale=0.5]
	\fillshade{1/1,2/2,3/3,4/4,5/5,6/6,7/7}
	\Dpath{0,0}{7}{7}{0,0,0,1,1,3,3,-1};
	\PFtext{0,0}{0/2,0/3,0/5,1/2,1/7,3/3,3/7};
	\PFad{8,0}{1/0/0,2/1/0,3/2/0,4/2/1,5/3/1,6/2/0,7/3/0};
	\end{tikzpicture}
	\caption{A $(7,7)$-parking function with area 13 and dinv 2.}
	\label{fig:areadinv}
\end{figure}

For a word parking function $\PF\in\WPFc_n$, we define the \emph{label weight} (or \emph{car weight}) of PF to be
$$
X^{\PF} := \prod_{i=1}^{n} x_{\ell_i(\PF)}.
$$
Then all statistics involved in the Shuffle Theorem have been defined.
The Delta Conjecture also requires the following combinatorial terminology about parking functions. For a parking function $\PF\in\WPFc_{n}$, 
we define 
\begin{eqnarray*}
	\valley(\PF) &:=& \{i : a_i(\PF)\leq a_{i-1}(\PF)\},\\
	\Rise(\PF) &:=& \{i : a_i(\PF) = a_{i-1}(\PF)+1\},\quad\mbox{and}\\
	\Val(\PF) &:=& \{i : a_i(\PF)< a_{i-1}(\PF)\mbox{ or }a_i(\PF)=a_{i-1}(\PF)\mbox{ and }\ell_i(\PF)>\ell_{i-1}(\PF)\}
\end{eqnarray*}
to be the sets of \emph{valleys}, 
\emph{double rises} and \emph{contractible valleys} of PF.

We denote the right hand sides of Equations (\ref{delta1}) and (\ref{delta2}) by $\Rise_{n,k}[X;q,t]$ and $\Val_{n,k}[X;q,t]$:
\begin{eqnarray*}
\Rise_{n,k}[X;q,t]&=&\sum_{\PF\in\WPFc_n}t^{\area(\PF)}q^{\dinv(\PF)} X^{\PF}
\prod_{i\in \Rise(\PF)}
\left(1+\frac{z}{t^{a_i(\PF)}} \right)\bigg|_{z^{n-k-1}}\\
\Val_{n,k}[X;q,t]&=&\sum_{\PF\in\WPFc_n}t^{\area(\PF)}q^{\dinv(\PF)} X^{\PF}
\prod_{i\in \Val(\PF)}
\left(1+\frac{z}{q^{d_i(\PF)+1}} \right)\bigg|_{z^{n-k-1}}.
\end{eqnarray*}

Consider the factor $t^{\area(\PF)}\prod_{i\in \Rise(\PF)}
\left(1+\frac{z}{t^{a_i(\PF)}} \right)\Big|_{z^{n-k-1}}$ in $\Rise_{n,k}[X;q,t]$. Each term in the expansion of this factor is a power of $t$, and the power is $\area(\PF)$ minus $n-k-1$ row-areas $a_i(\PF)$ of the double rise rows. Similarly, in the factor $q^{\dinv(\PF)}\prod_{i\in \Val(\PF)}
(1+\frac{z}{q^{d_i(\PF)+1}})\Big|_{z^{n-k-1}}$ in $\Val_{n,k}[X;q,t]$, each term is a power of $q$, and the power is $\dinv(\PF)$ minus $n-k-1$ row-dinvs $(d_i(\PF)+1)$ of the contractible valley rows. Thus, if we define
\begin{eqnarray*}
	\WPFc_{n,k}^{\Rise} &:=& \{(\PF,R): P\in\WPFc_n, R\subseteq \Rise(\PF), |R|=k  \},\\
	\WPFc_{n,k}^{\Val} &:=& \{(\PF,V): P\in\WPFc_n, V\subseteq \Val(\PF), |V|=k  \}
\end{eqnarray*}
and let 
\begin{eqnarray*}
	\area^-(\PF,R)&:=&\sum_{i\in [n]\backslash R} a_i(\PF),\\
	\dinv^-(\PF,V)&:=&\sum_{i\in [n]\backslash V} d_i(\PF) - |V|,
\end{eqnarray*}
then
\begin{eqnarray*}
	\Rise_{n,k}[X;q,t]&=&\sum_{(\PF,R)\in\WPFc_{n,n-k-1}^{\Rise}}t^{\area^-(\PF,R)}q^{\dinv(\PF)} X^{\PF},\\
	\Val_{n,k}[X;q,t]&=&\sum_{(\PF,V)\in\WPFc_{n,n-k-1}^{\Val}}t^{\area(\PF)}q^{\dinv^-(\PF,V)} X^{\PF},
\end{eqnarray*}
and the Delta Conjecture can be stated as
$$
\Delta'_{e_k}e_n=\Rise_{n,k}[X;q,t]=\Val_{n,k}[X;q,t]
$$
for any integers $n> k\geq 0$.
 
We call a pair $(\PF,R)\in\WPFc_{n,k}^{\Rise}$ (or $(\PF,V)\in\WPFc_{n,k}^{\Val}$) a \emph{rise-decorated} (or \emph{valley-decorated}) parking function, which can be seen as a parking function PF with $k$ rows in $\Rise(\PF)$ (or $\Val(\PF)$) marked with a star $\ast$. \fref{MWPF} shows examples of rise-decorated and valley-decorated parking functions.

\begin{figure}[ht]
	\centering	
	\begin{tikzpicture}[scale=0.5]
	\fillshade{1/1,2/2,3/3,4/4,5/5,6/6,7/7}
	\Dpath{0,0}{7}{7}{0,0,0,1,1,3,4,-1};
	\PFtext{0,0}{0/2,0/3,0/4,1/2,1/5,3/1,4/3};
	\fillsome{0/2/\ast,1/5/\ast}
	\end{tikzpicture}
	\qquad\qquad
	\begin{tikzpicture}[scale=0.5]
	\fillshade{1/1,2/2,3/3,4/4,5/5,6/6,7/7}
	\Dpath{0,0}{7}{7}{0,0,0,1,1,3,4,-1};
	\PFtext{0,0}{0/2,0/3,0/4,1/2,1/5,3/1,4/3};
	\fillsome{3/6/\ast,4/7/\ast}
	\end{tikzpicture}
	\caption{Examples: parking functions in $\WPFc_{7,2}^{\Rise}$ and $\WPFc_{7,2}^{\Val}$.}
	\label{fig:MWPF}
\end{figure}
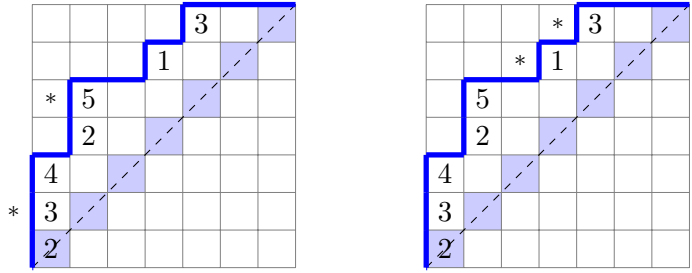

Now we shall give details of the Extended Delta Conjecture.
Given an $(n,n)$-Dyck path $P$, recall that the valley set of $P$ is defined to be
$$
\valley(P):=\{i:a_i(P)<a_{i-1}(P)\}.
$$
We say that a word-labeling of a Dyck path \emph{has $r$ blank valleys} if there are $r$ valleys not receiving a label. Such labeled Dyck paths are called \emph{extended word parking functions}. We let $\WPFc_{n;r}$ denote the set of extended word parking functions of size $n+r$ with $r$ blank valleys. \fref{blankvalley} shows an example of a parking function in the set $\WPFc_{5;2}$. 
\begin{figure}[ht]
	\centering	
	\begin{tikzpicture}[scale=0.5]
	\fillshade{1/1,2/2,3/3,4/4,5/5,6/6,7/7}
	\Dpath{0,0}{7}{7}{0,0,0,1,1,3,4,-1};
	\PFtext{0,0}{0/2,0/3,0/4,1/{\ },1/5,3/3};
	\PFad{8,0}{1/0/0,2/1/0,3/2/0,4/2/1,5/3/2,6/2/0,7/2/0};
	\end{tikzpicture}
	\caption{A $(7,7)$-extended parking function with $2$ blank valleys.}
	\label{fig:blankvalley}
\end{figure}
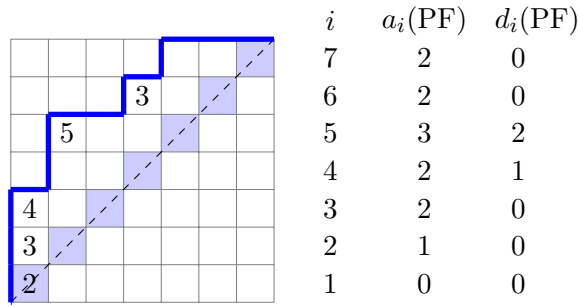

A more convenient way to draw an extended word parking function is that, we can fill the blank valleys with 0's, thus an extended word parking function is a parking function with labels in $\Z_{\geq 0}$ such that $0$ does not appear in the first row (since the first row is not a valley).

With 0's in the blank valley positions, we can define the area and dinv components $a_i(\PF)$ and $d_i(\PF)$ on each parking functions in $\WPFc_{n;r}$ in the same way. We still let $\Rise(\PF)=\{i: a_i(\PF) = a_{i-1}(\PF)+1\}$ denote the double rise set. For sake of labeling the blank valleys with 0's, we can define the contractible valley set $\Val(\PF)$ in the same way as normal word parking functions. 

Further, we can define the set of \emph{rise-decorated} (or \emph{valley-decorated}) \emph{parking functions with blank valleys}. The set of {rise-decorated} (or {valley-decorated}) {parking functions with $n$ cars, $r$ blank valleys} and $k$ marked double rises (or contractible valleys) is denoted by
$\WPFc_{r;n,k}^{\Rise}$ (or $\WPFc_{r;n,k}^{\Val}$).

We let $\Rise_{r;n,k}[X;q,t]$ denote the combinatorial side of \cref{deltar1} and $\Val_{r;n,k}[X;q,t]$ denote the combinatorial side of \cref{deltar2}. Notice that the combinatorial sides of the two conjectures could also be written as generating functions of the sets $\WPFc_{r;n,n-k-1}^{\Rise}$ and $\WPFc_{r;n,n-k-1}^{\Val}$, i.e.\ we have
\begin{eqnarray*}
	\Rise_{r;n,k}[X;q,t]&=&\sum_{(\PF,R)\in\WPFc_{r;n,n-k-1}^{\Rise}}t^{\area^-(\PF,R)}q^{\dinv(\PF)}x^\PF ,\\
	\Val_{r;n,k}[X;q,t]&=&\sum_{(\PF,V)\in\WPFc_{r;n,n-k-1}^{\Val}}t^{\area(\PF)}q^{\dinv^-(\PF,V)}x^\PF.
\end{eqnarray*}

\section{Extended ordered multiset partitions}
\subsection{Ordered set partitions and ordered multiset partitions}
Let $n\geq 0$ be any integer. A {\em set partition} $\pi$ of the set $[n]= \{1, \ldots, n\}$ is a family of nonempty, pairwise disjoint subsets $B_1,B_2,\ldots,B_k$ of $[n]$ called {\em parts} (or {\em blocks}) such that $\cup^k_{i=1}B_i=[n]$. We let $\ell(\pi)$ denote the number of parts in $\pi$ and 
$|\pi|=n$ denote the size of $\pi$.  
We let $\min(B_i)$ and $\max(B_i)$ denote the minimum and maximum elements 
of $B_i$ and we use the convention that we order the parts so that  
$\min(B_1)< \cdots < \min(B_k)$. To simplify notation, we 
shall write $\pi$ as $B_1/ \cdots /B_k$. Thus we would write 
$\pi =134/268/57$ for the set partition $\pi$ of $[8]$ with parts 
$B_1 = \{1,3,4\}$, $B_2 = \{2,6,8\}$ and $B_3=\{5,7\}$. 

An {\em ordered set partition} with underlying 
set partition $\pi$ is just a permutation of the parts of $\pi$, 
i.e.\ $\delta =B_{\sg_1}/ \cdots /B_{\sg_k}$ for some permutation $\sg$ in the 
symmetric group $\Sn{k}$. 
For example, $\delta =57/134/268$ is an ordered set partition 
of the set $[8]$ with underlying set partition $\pi =134/268/57$. 

Let $\pi=B_1/ \cdots /B_k$ be an ordered set partition of $[n]$. The strong composition $\la(\pi)=(|B_1|,\ldots,|B_k|)$ is called the \emph{shape} of $\pi$.
We let $\OP_n$ denote the set of ordered set partitions of $[n]$, and $\OP_{n,k}$ denote the set of ordered set partitions of $[n]$ with $k$ parts. Further, we let $\OP_{n,\al}$ denote the set of ordered set partitions of $[n]$ with shape $\al$.

More generally, for a weak composition $\beta = \beta_1 \cdots \beta_\ell \vDash n$, an \emph{ordered multiset partition} with \emph{content} $\beta$ is defined to be a partition of the multiset $A(\beta)=\{i^{\beta_i}:1\leq i\leq \ell\}$ into several ordered sets called \emph{blocks} where \emph{repetition is not allowed} in each block. We denote the set of ordered multiset partitions with content $\beta$ by $\OP_{\beta}$. Similar, we have $\OP_{\beta,k}$ and $\OP_{\beta,\al}$. For example, $\pi=234/26/123$ is an ordered multiset partition in $\OP_{(1,3,2,1,0,1),(3,2,3)}$.

We shall define 4 statistics: inv, maj, dinv and minimaj on ordered multiset partitions.

Given $\pi=B_1/\cdots/B_k\in\OP_{\beta,k}$, the \emph{inversion} statistic $\inv(\pi)$ is defined to be the number of pairs $a>b$ such that $b$ is the minimum of its block, and $a$ is in some block that is strictly left of $b$'s block. Such pairs are called \emph{inversion pairs}. For example, $\pi =134/268/57$ has 4 inversions, and the inversion pairs are $(3,2),(4,2),(6,5),(8,5)$.

For an ordered partition $\pi=B_1/\cdots/B_k\in\OP_{\beta,k}$, let $B_i^h$ denote the \thn{h} smallest element in part $B_i$, then the \emph{diagonal inversion} of $\pi$ is defined to be 
$$
\dinv(\pi):=|\{(h,i,j):i<j, B_i^h>B_j^h\}\cup\{(h,i,j):i<j, B_i^h>B_j^{h+1}\}|,
$$
where the triples in the left set are called \emph{primary dinvs}, and the triples in the right set are called \emph{secondary dinvs}. For example, $\pi =134/268/57$ has 4 dinvs, which are all secondary dinvs: $(1,1,2),(1,1,3),(1,2,3),(2,1,2)$.

We let $\sg=\sg(\pi)$ of a partition $\pi\in\OP_{\beta,k}$ be the word obtained by writing each block $B_i$ in decreasing order for $i=1\cdots k$. We also define the index word $\ind(\pi) = 0^{|B_1|} 1^{|B_2|}\cdots(k-1)^{|B_k|}$. Then the \emph{major index} of $\pi$ is
$$
\maj(\pi) := \sum_{i:\sg_i>\sg_{i+1}} \ind(\pi)_{i+1}.
$$
For example, if $\pi=134/268/57$, then $\sg=43186275$, $\ind(\pi)=00011122$ and $\maj(\pi)=0+0+1+1+2=4$.

Given $\pi = B_1/\cdots/B_k \in\OP_{\beta,\al}$ where $\al=(\al_1,\ldots,\al_k)$, we first construct a word $\miniword(\pi)$ by organizing the elements in each block and list the organized blocks $B_1,\ldots,B_k$. We first organize the numbers in $B_k$ in increasing order. Then suppose that we have processed block $B_{i+1}$, we shall organize the numbers in $B_i$ by placing the numbers strictly bigger than the first number of $B_{i+1}$ first in increasing order, followed by the remaining numbers also in increasing order, then we place the organized numbers on the left of the existing sequence. For example, if $\pi=2/34/13/13/2$, then $\miniword(\pi)=23413312$. The \emph{minimum major index} of $\pi$ is defined by
$$
\minimaj(\pi):=\maj(\miniword(\pi)).
$$

The four statistics are closely related to the Delta Conjecture. Let 
$$
D_{\beta,k}^{\stat}(q) := \sum_{\pi\in\OP_{\beta,k}}q^{\stat(\pi)}
$$
where stat is one of the statistics \emph{inv, maj, dinv, minimaj},
Haglund, Remmel and the second author in \cite{HRW} proved that
\begin{theorem}[Haglund, Remmel and Wilson]\label{theorem:combo}
	For any integers $n,k$ and weak composition $\beta$,
	\begin{eqnarray}
	\Rise_{n,k}[X;q,0]|_{M_\beta} &=& D_{\beta,k+1}^{\dinv}(q),\\
	\Rise_{n,k}[X;0,q]|_{M_\beta} &=& D_{\beta,k+1}^{\maj}(q),\\
	\Val_{n,k}[X;q,0]|_{M_\beta} &=& D_{\beta,k+1}^{\inv}(q),\\
	\Val_{n,k}[X;0,q]|_{M_\beta} &=& D_{\beta,k+1}^{\minimaj}(q).
	\end{eqnarray}
\end{theorem}
They proved \tref{combo} by constructing 4 bijections of the form $\ga^{\stat}$ for $\stat=\dinv,$ $\maj,\inv$ and $\minimaj$ between ordered multiset partitions and word parking functions. We present the four bijections in Appendix A.
It is a \textbf{fact} that for any ordered multiset partition $\pi$, \emph{each bijection $\ga^{\stat}$ maps the the minimum element in the last part of $\pi$ to the car in the first row in the parking function $\ga^{\stat}(\pi)$}, mentioned in Appendix A. We are going to use the fact when we prove \tref{combor}.

On the combinatorial side, 
the second author \cite{wilson} and Rhoades \cite{Rhoades} proved the following theorem:
\begin{theorem}[Rhoades and
	Wilson]\label{theorem:equi}
	For any integers $n,k$,
	\begin{equation}
	\Rise_{n,k}(X;q,0)=\Rise_{n,k}(X;0,q)=\Val_{n,k}(X;q,0)=\Val_{n,k}(X;0,q).
	\end{equation}
\end{theorem}

\subsection{Extended permutations, extended ordered set and multiset partitions}

We shall generalize the definitions of permutations, ordered set partitions and ordered multiset partitions in the way that the number 0 is allowed to be an entry.

Let $\beta = \{\beta_1,\ldots,\beta_\ell\}\vDash n$ be a weak composition and $A(\beta)=\{i^{\beta_i}:1\leq i\leq \ell\}$ be its corresponding multiset. A permutation of $A(\beta)$ is an ordering of the entries in the multiset $A(\beta)$. We let $\Sn{\beta}$ denote the set of permutations of $A(\beta)$.

Given a weak composition $\beta\vDash n$ and an integer $r\geq 0$, an \emph{extended permutation} (or a \emph{tail positive permutation}) is a permutation of the multiset $A(\beta)\cup \{0^r\}$ such that the last entry is not $0$. We let $\Sn{r;\beta}$ denote the set of extended permutations of $A(\beta)\cup \{0^r\}$. Clearly, $\Sn{0;\beta}=\Sn{\beta}$.

In a similar way, one can define extended ordered set and multiset partitions. We let $\OP_{1;n}$ denote the set of \emph{extended ordered set partitions}, which are ordered set partitions of the set $\{0\}\cup\{1,\ldots,n\}$ such that the number $0$ is not contained in the last block. Similar to the definition of $\OP_{n,k}$ and $\OP_{n,\al}$, we have $\OP_{1;n,k}$ and $\OP_{1;n,\al}$.

An \emph{extended ordered multiset partition} with content $\beta\vDash n$ with $r$ 0's is an ordered multiset partition of the set $A(\beta)\cup \{0^r\}$  such that $0$ is not contained in the last block. We let $\OP_{r;\beta}$ denote the set of all such extended ordered multiset partitions. Similarly, we have $\OP_{r;\beta,k}$ and $\OP_{r;\beta,\al}$.

The above three new combinatorial objects are defined from the same idea that they do not end with $0$, and extended ordered multiset partitions have nice combinatorial properties. It is easy to check that all the four statistics: inv, maj, dinv, minimaj are well defined on the set $\OP_{r;\beta,\al}$. Let 
$$
D_{r;\beta,k}^{\stat}(q) := \sum_{\pi\in\OP_{r;\beta,k}}q^{\stat(\pi)}
$$
where stat is one of the statistics \emph{inv, maj, dinv, minimaj}. 
We can prove the following theorem:
\begin{theorem}\label{theorem:combor}
	For any integers $n,k,r$ and weak composition $\beta$,
	\begin{eqnarray}
	\Rise_{r;n,k}[X;q,0]|_{M_\beta} &=& D_{r;\beta,k+1}^{\dinv}(q),\\
	\Rise_{r;n,k}[X;0,q]|_{M_\beta} &=& D_{r;\beta,k+1}^{\maj}(q),\\
	\Val_{r;n,k}[X;q,0]|_{M_\beta} &=& D_{r;\beta,k+1}^{\inv}(q),\\
	\Val_{r;n,k}[X;0,q]|_{M_\beta} &=& D_{r;\beta,k+1}^{\minimaj}(q).
	\end{eqnarray}
\end{theorem}
\begin{proof}
	Similar to the definition of $\OP_{r;\be}$, we shall let $\OPa_{r;\be}$ denote the set of ordered multiset partitions of the set $A(\be)\cup \{0^r\}$, but there is no restriction of the placement of 0 (i.e.\ 0 is allowed to be in the last block). Similarly, we have $\OPa_{r;\beta,k}$ and $\OPa_{r;\beta,\al}$.

	Haglund et al.\ proved \tref{combo} by constructing 4 bijections $\ga^\dinv, \ga^\maj, \ga^\inv, \ga^\minimaj$ between ordered multiset partitions and decorated word parking functions:
	\begin{eqnarray*}
		\ga^\dinv:\OP_{\beta,k+1} &\rightarrow& \{(\PF,R)\in\WPFc_{n,n-k-1}^{\Rise},\ 
		X^{\PF} = \prod_{i=1}^{\ell(\be)}x_i^{\be_i},\ \area^-(\PF,R)=0\},
		\\
		\ga^\maj:\OP_{\beta,k+1} &\rightarrow& \{(\PF,R)\in\WPFc_{n,n-k-1}^{\Rise},\ 
		X^{\PF} = \prod_{i=1}^{\ell(\be)}x_i^{\be_i},\ \dinv(\PF)=0\},
		\\
		\ga^\inv:\OP_{\beta,k+1} &\rightarrow& \{(\PF,V)\in\WPFc_{n,n-k-1}^{\Val},\ 
		X^{\PF} = \prod_{i=1}^{\ell(\be)}x_i^{\be_i},\ \area(\PF)=0\},
		\\
		\ga^\minimaj:\OP_{\beta,k+1} &\rightarrow& \{(\PF,V)\in\WPFc_{n,n-k-1}^{\Val},\ 
		X^{\PF} = \prod_{i=1}^{\ell(\be)}x_i^{\be_i},\ \dinv^-(\PF,V)=0\}.
	\end{eqnarray*}
	The details can be found in Appendix A. If we allow 0 as an element of an ordered multiset partition, then the four maps can be naturally generalized to the set $\OPa_{r;\beta,k}$, and the range of the maps are parking functions that allow 0 as a car, i.e.\ if we let $\WPFc_{r;n,k}^{\Rise+}$ and $\WPFc_{r;n,k}^{\Val+}$ be the set of rise and valley decorated word parking function with $r$ 0's (car 0 is allowed in the first row), then
	we have bijections
	\begin{eqnarray*}
		\ga^\dinv:\OPa_{r;\beta,k+1} &\rightarrow& \{(\PF,R)\in\WPFc_{r;n,n-k-1}^{\Rise+},\ 
		X^{\PF} = \prod_{i=1}^{\ell(\be)}x_i^{\be_i},\ \area^-(\PF,R)=0\},
		\\
		\ga^\maj:\OPa_{r;\beta,k+1} &\rightarrow& \{(\PF,R)\in\WPFc_{r;n,n-k-1}^{\Rise+},\ 
		X^{\PF} = \prod_{i=1}^{\ell(\be)}x_i^{\be_i},\ \dinv(\PF)=0\},
		\\
		\ga^\inv:\OPa_{r;\beta,k+1} &\rightarrow& \{(\PF,V)\in\WPFc_{r;n,n-k-1}^{\Val+},\ 
		X^{\PF} = \prod_{i=1}^{\ell(\be)}x_i^{\be_i},\ \area(\PF)=0\},
		\\
		\ga^\minimaj:\OPa_{r;\beta,k+1} &\rightarrow& \{(\PF,V)\in\WPFc_{r;n,n-k-1}^{\Val+},\ 
		X^{\PF} = \prod_{i=1}^{\ell(\be)}x_i^{\be_i},\ \dinv^-(\PF,V)=0\}.
	\end{eqnarray*}
	
	We have mentioned the fact below \tref{combo} and in Appendix A that each bijection $\ga^{\stat}$ maps the minimum element in the last part of $\pi$ into the car in the first row of $\ga^{\stat}(\pi)$. 
	Since the set $\OP_{r;\beta,k}$ contains ordered multiset partitions in $\OPa_{r;\beta,k}$ that 0 is not contained in the last block, the restriction of the maps $\ga^{\stat}$ on the set $\OP_{r;\beta,k}\subseteq\OPa_{r;\beta,k}$ is a bijection between $\OP_{r;\beta,k}$ and the corresponding set of parking functions with $r$ 0's but 0 is not allowed in the first row, which exactly matches the set $\WPFc_{r;n,n-k-1}^{\Rise}$ or $\WPFc_{r;n,n-k-1}^{\Val}$, and the restriction of the maps $\ga^{\stat}$ on $\OP_{r;\beta,k}\subseteq\OPa_{r;\beta,k}$ are bijections:
	\begin{eqnarray*}
		\ga^\dinv:\OP_{r;\beta,k+1} &\rightarrow& \{(\PF,R)\in\WPFc_{r;n,n-k-1}^{\Rise},\ 
		X^{\PF} = \prod_{i=1}^{\ell(\be)}x_i^{\be_i},\ \area^-(\PF,R)=0\},
		\\
		\ga^\maj:\OP_{r;\beta,k+1} &\rightarrow& \{(\PF,R)\in\WPFc_{r;n,n-k-1}^{\Rise},\ 
		X^{\PF} = \prod_{i=1}^{\ell(\be)}x_i^{\be_i},\ \dinv(\PF)=0\},
		\\
		\ga^\inv:\OP_{r;\beta,k+1} &\rightarrow& \{(\PF,V)\in\WPFc_{r;n,n-k-1}^{\Val},\ 
		X^{\PF} = \prod_{i=1}^{\ell(\be)}x_i^{\be_i},\ \area(\PF)=0\},
		\\
		\ga^\minimaj:\OP_{r;\beta,k+1} &\rightarrow& \{(\PF,V)\in\WPFc_{r;n,n-k-1}^{\Val},\ 
		X^{\PF} = \prod_{i=1}^{\ell(\be)}x_i^{\be_i},\ \dinv^-(\PF,V)=0\}.
	\end{eqnarray*}
	\tref{combor} follows from the fact that $\ga^{\stat}$ maps the statistic $\stat$ into parking function statistics $\dinv,\area^-,\dinv^-,\area$.
\end{proof}

Thus, the combinatorial sides of the conjectures about the expression $\Delta'_{e_k} \Delta_{h_r}e_n$ at the case when $q$ or $t$ equals $0$ become generating functions about extended ordered multiset partitions. We shall show in the following two sections that the statistics inv, maj, dinv, minimaj are equi-distributed on $\OP_{r;\beta,k}$.

\section{The identity $D_{r;\beta,k}^{\dinv}(q)
	=D_{r;\beta,k}^{\maj}(q)
	=D_{r;\beta,k}^{\inv}(q)$}

Recall that we let $\OPa_{r;\be}$ denote the set of ordered multiset partitions of the set $A(\be)\cup \{0^r\}$ and 0 is allowed to be in the last block. We also have $\OPa_{r;\beta,k}$ and $\OPa_{r;\beta,\al}$.

In fact, $\OPa_{r;\beta,k}$ only enlarges the alphabet of $\OP_{\beta,k}$ from $\Z_+$ to $\Z_{\geq 0}$, and it will inherit all the properties of $\OP_{\beta,k}$. For a composition $\be=(\be_1,\ldots,\be_n)$ and integers $r,k$, we let $$
D_{r;\beta,k}^{\stat +}(q) := \sum_{\pi\in\OPa_{r;\beta,k}}q^{\stat(\pi)}
$$
where stat is one of the statistics \emph{inv, maj, dinv, minimaj},
then clearly
\begin{equation*}
D_{r;(\be_1,\ldots,\be_n),k}^{\stat +}(q) =D_{(r,\be_1,\ldots,\be_n),k}^{\stat}(q),
\end{equation*}
since we can add 1 to all the entries of an ordered multiset partition in $\OPa_{r;(\be_1,\ldots,\be_n),k}$ to get an ordered multiset partition in $\OP_{(r,\be_1,\ldots,\be_n),k}$.
It follows from \tref{equi} that,
\begin{corollary}\label{corollary:Dplus}
	For any integers $n,r$ and composition $\be$,
	$$
	D_{r;\be,k}^{\inv +}(q) = D_{r;\be,k}^{\maj +}(q) = 
	D_{r;\be,k}^{\dinv +}(q) = D_{r;\be,k}^{\minimaj +}(q).
	$$
\end{corollary}

For a composition $\be=(\be_1,\ldots,\be_n)$, we let $\be^-=(\be_1,\ldots,\be_{n-1})$ be the composition obtained by removing the last part of $\be$. We also let $[0,\ell]$ be the set $\{0,1,\ldots,\ell\}$. 
Further, for a set $S$, we let $\binom{S}{k}$ be the set of size-$k$ subsets of $S$, and 
$\multiset{S}{k}$
be the set of size-$k$ multisets with elements in $S$.

In order to prove the result about ordered multiset partition that $
D_{\be,k}^{\inv}(q) = D_{\be,k}^{\maj}(q) = 
D_{\be,k}^{\dinv}(q)
$,
the second author in \cite{wilson} constructed 3 \emph{insertion} maps:
$$
\phi^{\stat}_{\be,k,\ell}: \OP_{\be^-,\ell} \times \binom{[0,\ell-1]}{\be_n-k+\ell} \times \multiset{[0,\ell]}{k-\ell} \rightarrow \OP_{\be,k},
$$
where stat is one of the statistics \emph{inv, maj, dinv}, and he proved that
$$
\stat\left(\phi^{\stat}_{\be,k,\ell}(\pi,U,B)\right) = \stat(\pi)+\sum_{u\in U} u+\sum_{b\in B} b
$$
for all the three statistics. In this section, we shall generalize the insertion maps in \cite{wilson} to extended ordered multiset partitions to prove the identity that
$$D_{r;\beta,k}^{\dinv}(q)
=D_{r;\beta,k}^{\maj}(q)
=D_{r;\beta,k}^{\inv}(q).$$
This identity is also proved by D'Adderio, Iraci and Wyngaerd in \cite{Michele} independently.

\subsection{The insertion map for inv}

We shall generalize the map $\phi^{\inv}_{\be,k,\ell}$ in \cite{wilson} to the extended case as
\begin{align*}
\phi^{\inv}_{r;\be,k,\ell}&: 
\OP_{r;\be^-,\ell} \times \binom{[0,\ell-1]}{\be_n-k+\ell} \times \multiset{[0,\ell]}{k-\ell} \\
&+ \left(\OPa_{r;\be^-,\ell} - \OP_{r;\be^-,\ell}\right) \times \binom{[0,\ell-1]}{\be_n-k+\ell} \times \multiset{[0,\ell]}{k-\ell-1} 
\rightarrow \OP_{r;\be,k}
\end{align*}
such that
\begin{equation}\label{inv}
\inv\left(\phi^{\inv}_{r;\be,k,\ell}(\pi,U,B)\right) = \inv(\pi)+\sum_{u\in U} u+\sum_{b\in B} b.
\end{equation}

Given $\be=(\be_1,\ldots,\be_n)$ and $\pi\in\OP_{r;\be^-,\ell}$, we label each block plus the space to the left of $\pi$ from right to left with numbers $0,1,\ldots,\ell$. Then for any $U\in \binom{[0,\ell-1]}{\be_n-k+\ell}$ and $B\in \multiset{[0,\ell]}{k-\ell}$, we construct $\phi^{\inv}_{r;\be,k,\ell}(\pi,U,B)$ as follows.

We repeatedly remove the largest number $i$ from the multiset $U\cup B$, taking from $U$ first if the largest numbers are equal. If $i\in U$, then we place an $n$ to the block with label $i$; if $i\in B$, then we add a new block of a singleton $n$ to the right of the block with label $i$. This process constructs all the ordered multiset partitions in $\OP_{r;\be,k}$ such that 
the last block which is not a singleton $n$ does not contain a 0.

In order to construct the remaining ordered partitions in $\OP_{r;\be,k}$, those whose 
last block which is not a singleton $n$ does not contain a 0,
we take ordered multiset partitions $\pi$ in the set $\left(\OPa_{r;\be^-,\ell} - \OP_{r;\be^-,\ell}\right)$ 
(which means the last block of $\pi$ contains 0). Then for any $U\in \binom{[0,\ell-1]}{\be_n-k+\ell}$ and $B'\in \multiset{[0,\ell]}{k-\ell-1}$, we set the multiset 
$B=B'\cup \{0\}$, and we construct $\phi^{\inv}_{r;\be,k,\ell}(\pi,U,B)$ by repeatedly inserting numbers in the multiset $U\cup B$ in the same way. The 0 in $B$ ensures that the result will no longer have any 0's in its rightmost block. One can check easily that this gives all the ordered multiset partitions in $\OP_{r;\be,k}$, and the inv statistic increases by $i$ each time we insert an $i$, thus Equation (\ref{inv}) follows.

For example, suppose that $r=2$, $\beta = (2,3,2,4)$, and $k = 5$, $\ell=3$, and 
\begin{align*}
\pi &= 23/0123/012 \in\left(\OPa_{r;\be^-,\ell} - \OP_{r;\be^-,\ell}\right), \\
U &= \{0, 2\} \in \binom{[0,\ell-1]}{\be_n-k+\ell}, \\
B' &= \{3\} \in \multiset{[0,\ell]}{k-\ell-1}.
\end{align*}
Block 012 is labeled 0, block 0123 is labeled 1, and block 23 is labeled 2. The space to the far left receives label 3. First we take $i=3$ from $B$ and we insert a new singleton 4 block at the far left, yielding $4/23/0123/012$. Next we take $i=2$ from $U$, so we insert a 4 into the the 23 block and get $4/234/0123/012$. Then we take $i=0$ from $U$ and obtain $4/234/123/0124$. Finally, since $B = B' \cup \{0\}$, we insert a singleton 4 block at the far right to get $4/234/123/0124/4$.

\subsection{The insertion map for maj}

In order to define the map $\phi^{\maj}_{r;\be,k,\ell}$, we shall introduce the \emph{descent-starred permutation notation} of an ordered partition. For any ordered partition $\pi = B_1/\cdots/B_n$, we write the numbers of each block in decreasing order, remove the slashes and add stars at the descent positions that are entirely contained in some block of $\pi$. This permutation with stars is called  the \emph{descent-starred permutation notation} of $\pi$. 

The set of positions with stars is denoted by $S(\pi)$, and the permutation is denoted by $\sg(\pi)$ introduced in Section 3.1. For example, if $\pi = 134/47/23$, then $\sg(\pi)=4317432$, $S(\pi)=\{1,2,4,6\}$ and $4_\ast 3_\ast 17_\ast 43_\ast 2$ is the corresponding descent-starred permutation.

The map
\begin{align*}
\phi^{\maj}_{r;\be,k,\ell}&: 
\OP_{r;\be^-,\ell} \times \binom{[0,\ell-1]}{\be_n-k+\ell} \times \multiset{[0,\ell]}{k-\ell} \\
&+ \left(\OPa_{r;\be^-,\ell} - \OP_{r;\be^-,\ell}\right) \times \binom{[0,\ell-1]}{\be_n-k+\ell} \times \multiset{[0,\ell]}{k-\ell-1} 
\rightarrow \OP_{r;\be,k}
\end{align*}
is defined as follows. 

Given $\be=(\be_1,\ldots,\be_n)$ and $\pi\in\OP_{r;\be^-,\ell}$, we write $\pi$ in descent-starred notation and let $\sg=\sg(\pi)$. 
With labels $0,\ldots,\ell$, we first label the rightmost position, then the unstarred descent positions of $\pi$  from right to left, then the unstarred non-descent positions (including the leftmost position) from left to right.

For any $U\in \binom{[0,\ell-1]}{\be_n-k+\ell}$ and $B\in \multiset{[0,\ell]}{k-\ell}$, we construct $\phi^{\maj}_{r;\be,k,\ell}(\pi,U,B)$ by setting $U^+ = \{u+1:u\in U\}$, then repeatedly remove the largest number $i$ from the multiset $U^+ \cup B$, taking from $B$ first if the largest numbers are equal. The algorithm of inserting $i$ is as follows:
\begin{enumerate}
	\item Insert the number $n$ at the position with label $i$.
	\item Move each star that appears to the right of the new $n$ one descent to the left.
	\item If $i\in U^+$, then star the rightmost descent.
	\item Relabel the starred permutation as before, stopping at $i$ if $i\in B$ and $i-1$ if $i\in U^+$.
\end{enumerate}
This process constructs all the ordered multiset partitions in $\OP_{r;\be,k}$ such that 
the last block which is not a singleton $n$ does not contain a 0.

In order to construct the remaining ordered partitions in $\OP_{r;\be,k}$, we take ordered multiset partitions $\pi$ in the set $\left(\OPa_{r;\be^-,\ell} - \OP_{r;\be^-,\ell}\right)$ such that the last block contains 0. Then for any $U\in \binom{[0,\ell-1]}{\be_n-k+\ell}$ and $B'\in \multiset{[0,\ell]}{k-\ell-1}$, we set $U^+ = \{u+1:u\in U\}$ and 
$B=B'\cup \{0\}$, and we construct $\phi^{\maj}_{r;\be,k,\ell}(\pi,U,B)$ by repeatedly inserting numbers in the multiset $U^+\cup B$ in the same way. One can check easily that this gives all the ordered multiset partitions in $\OP_{r;\be,k}$. The second author \cite{wilson} gave a proof that the maj statistic increases by $i$ each time we insert an $i$ in the non-extended case, which works naturally for the extended case, thus we have
\begin{equation}\label{maj}
\maj\left(\phi^{\maj}_{r;\be,k,\ell}(\pi,U,B)\right) = \maj(\pi)+\sum_{u\in U} u+\sum_{b\in B} b.
\end{equation}


Consider again the example $r=2$, $\beta = (2,3,2,4)$, and $k = 5$, $\ell=3$, and 
\begin{align*}
\pi &= 23/0123/012 \in\left(\OPa_{r;\be^-,\ell} - \OP_{r;\be^-,\ell}\right) ,\\
U &= \{0, 2\} \in \binom{[0,\ell-1]}{\be_n-k+\ell}, \\
B' &= \{3\} \in \multiset{[0,\ell]}{k-\ell-1}.
\end{align*}
As a descent-starred permutation, we write $\pi$ as $\sg(\pi) = 3_{\ast}23_{\ast}2_{\ast}1_{\ast}02_{\ast}1_{\ast}0$. The labeling of $\sg(\pi)$ is 
\begin{align*}
_1 3_{\ast}2_2 3_{\ast}2_{\ast}1_{\ast}0_3 2_{\ast}1_{\ast}0_0
\end{align*}
We take a 3 from $B$ and, after inserting a 4 at position 3 and shifting stars to the left, we get
\begin{align*}
3_{\ast}23_{\ast}2_{\ast}1_{\ast}04_{\ast}2_{\ast}10
\end{align*}
increasing the major index by 3. We relabel and continue with a 3 from $U^{+}$, obtaining
\begin{align*}
3_{\ast}2 4_{\ast} 3_{\ast}2_{\ast}1_{\ast}04_{\ast}21_{\ast}0
\end{align*}
adding a new star after the last descent since this 3 
comes
from $U^{+}$. We take the 1 from $U^{+}$ and get 
\begin{align*}
3_{\ast}2 4_{\ast} 3_{\ast}2_{\ast}1_{\ast}04_{\ast}24_{\ast}1_{\ast}0 .
\end{align*}
Finally, we take the 0 we added to $B$ to obtain 
\begin{align*}
3_{\ast}2 4_{\ast} 3_{\ast}2_{\ast}1_{\ast}04_{\ast}24_{\ast}1_{\ast}04 =   23 / 01234 / 24 /  014 / 4 \in \OP_{r;\be,k} .
\end{align*}

\subsection{The insertion map for dinv}
We define a map
\begin{multline*}
\phi^{\dinv}_{r;\be,k,\ell}: 
\OP_{r;\be^-,\ell} \times \binom{[0,\ell-1]}{\be_n-k+\ell} \times \multiset{[0,\ell]}{k-\ell} \\
+ \left(\OPa_{r;\be^-,\ell} - \OP_{r;\be^-,\ell}\right) \times \binom{[0,\ell-1]}{\be_n-k+\ell} \times \multiset{[0,\ell]}{k-\ell-1} 
\rightarrow \OP_{r;\be,k}.
\end{multline*}

Given $\be=(\be_1,\ldots,\be_n)$ and $\pi\in\OP_{r;\be^-,\ell}$, we label the $\ell+1$ spaces (the spaces between parts as well as the spaces in the two ends) of $\pi$ from right to left with numbers $0,1,\ldots,\ell$ which we call the \emph{gap labels}. Next, we label the blocks from  highest to  lowest length (from left to right for each length) with numbers $0,1,\ldots,\ell-1$ which we call the \emph{block labels}.

For any $U\in \binom{[0,\ell-1]}{\be_n-k+\ell}$ and $B\in \multiset{[0,\ell]}{k-\ell}$, we can construct $\phi^{\dinv}_{r;\be,k,\ell}(\pi,U,B)$ by inserting an $n$ into each block whose label is in $U$ and inserting a singleton block $\{n\}$ at the gap $b$ for each $b\in B$. This process constructs all the ordered multiset partitions in $\OP_{r;\be,k}$ such that
the last block which is not a singleton $n$ does not contain a 0.

In order to construct the remaining ordered partitions in $\OP_{r;\be,k}$, we take ordered multiset partitions $\pi$ in $\left(\OPa_{r;\be^-,\ell} - \OP_{r;\be^-,\ell}\right)$. Then for any $U\in \binom{[0,\ell-1]}{\be_n-k+\ell}$ and $B'\in \multiset{[0,\ell]}{k-\ell-1}$, we set the multiset $B=B'\cup \{0\}$, and we construct $\phi^{\dinv}_{r;\be,k,\ell}(\pi,U,B)$ in the same way. One can check easily that this gives all the ordered multiset partitions in $\OP_{r;\be,k}$, and the dinv statistic increases by $i$ each time we insert an $i$, thus we have
\begin{equation}\label{dinv}
\dinv\left(\phi^{\dinv}_{r;\be,k,\ell}(\pi,U,B)\right) = \dinv(\pi)+\sum_{u\in U} u+\sum_{b\in B} b.
\end{equation}

Consider once more the example $r=2$, $\beta = (2,3,2,4)$, and $k = 5$, $\ell=3$, and 
\begin{align*}
\pi &= 23/0123/012 \in\left(\OPa_{r;\be^-,\ell} - \OP_{r;\be^-,\ell}\right), \\
U &= \{0, 2\} \in \binom{[0,\ell-1]}{\be_n-k+\ell} ,\\
B' &= \{3\} \in \multiset{[0,\ell]}{k-\ell-1}.
\end{align*}
We take a 3 from $B$ and insert a singleton 4 to the far left, obtaining $4 / 23 / 0123/ 012$. Then we take a 2 from $U$ and add a 4 to the 23 block to get $4 / 234 / 0123 / 012$. We take a 0 from $U$ and add a 4 to the 0123 block to get $4 /234 / 01234 / 012$. Finally, we take the 0 we added to $B$ and add a 4 to the far right to obtain $4 /234 / 01234 / 012 / 4 \in \OP_{r; \be, k}$.

According to the definitions of maps $\phi^{\inv}_{r;\be,k,\ell}$, $\phi^{\maj}_{r;\be,k,\ell}$, $\phi^{\dinv}_{r;\be,k,\ell}$ and Equations (\ref{inv}), (\ref{maj}) and (\ref{dinv}), one can conclude the following.
\begin{theorem}\label{theorem:eq1}
	For any integers $n,r$ and composition $\be$,
	$$
	D_{r;\be,k}^{\inv}(q) = D_{r;\be,k}^{\maj}(q) = 
	D_{r;\be,k}^{\dinv}(q).
	$$
\end{theorem}
\noindent
We mention the common recursion shared by these polynomials in Section 6.

\section{The identity $D_{r;\beta,k}^{\inv}(q)
	=D_{r;\beta,k}^{\minimaj}(q)$}

The goal of this section is to generalize the $(\inv, \minimaj)$ equi-distribution theorem of Rhoades \cite{Rhoades} from the set $\OP_{\be,k}$ to the set $\OP_{r;\be,k}$. 
For our convenience, we shall abbreviate $D^{\inv}$ and $D^{\minimaj}$ to $I$ and $M$, i.e. we  shall use the notations
\begin{eqnarray*}
	&&I_{\be,k}(q) = D_{\beta,k}^{\inv}(q),\quad   
	I_{\be,\al}(q) = D_{\beta,\al}^{\inv}(q), \quad
	I_{r;\be,k}(q) = D_{r;\beta,k}^{\inv}(q),  \quad 
	I_{r;\be,\al}(q) = D_{r;\beta,\al}^{\inv}(q), \\
	&&M_{\be,k}(q) = D_{\beta,k}^{\minimaj}(q),   \quad
	M_{\be,\al}(q) = D_{\beta,\al}^{\minimaj}(q),\\
	&&M_{r;\be,k}(q) = D_{r;\beta,k}^{\minimaj}(q),   \quad
	M_{r;\be,\al}(q) = D_{r;\beta,\al}^{\minimaj}(q).
\end{eqnarray*}
Further, we let
\begin{eqnarray*}
	&&\Ia_{r;\be,k}(q) = D_{r;\beta,k}^{\inv+}(q),   \quad
	\Ia_{r;\be,\al}(q) = D_{r;\beta,\al}^{\inv+}(q) \quad \\
	&& \Ma_{r;\be,k}(q) = D_{r;\beta,k}^{\minimaj+}(q), \quad \mbox{and}\quad
	\Ma_{r;\be,\al}(q) = D_{r;\beta,\al}^{\minimaj+}(q) 
\end{eqnarray*}
denote the generating functions that allow 0 in the last block.

\subsection{The recursion for inv}
For any integer $m$ and set $S\subseteq [m]$, we let $\chi_S = (\chi_S(1),\ldots,\chi_S(m))$ be the sequence such that $\chi_S(i)=\chi(i\in S)$
where $\chi$ of a statement is 1 if the statement is true, 0 if false.
For two sequences $\ga_1$ and $\ga_2$ of the same length, we write $\ga_1\leq\ga_2$ if each entry of $\ga_1$ is less than or equal to the corresponding entry of $\ga_2$.

Given an integer $n$,  a weak composition $\be=(\be_1,\ldots,\be_m)\vDash n$ and a strong composition $\al=(\al_1,\ldots,\al_k)\svDash n$, we still use the notation $\al^-=(\al_1,\ldots,\al_{k-1})$ for the composition of $n-\al_k$ that the last part of $\al$ is removed.

Recall that by definition, $\OP_{r;\be,\al}$ is the set of extended ordered multiset partition of the multiset $A(\be)\cup \{0^r\}$ and shape $\al$ such that $0$ is not contained in the last block, while $\OPa_{r;\be,\al}$ allows $0$ in the last block. Their generating functions tracking the statistic inv are $I_{r;\be,\al}(q)$ and $\Ia_{r;\be,\al}(q)$ respectively. Then we have the following theorem which is analogous to Lemma 3.2 in \cite{Rhoades}.
\begin{theorem}\label{theorem:Ialpha}
	The generating function $I_{r;\be,\al}(q)$ satisfies the following equation:
	\begin{equation}\label{eqI}
	I_{r;\be,\al}(q) = \sum_{\substack{S\subseteq[m],\ |S|=\al_k,\\\chi_S\leq \beta}}
	q^{\sum_{i=min(S)+1}^{m}(\be_i-\chi_S(i))} \Ia_{r;\be-\chi_S,\al^-}(q).
	\end{equation}
\end{theorem}
\begin{proof}
	Consider an ordered multiset partition $\mu=B_1/\cdots/B_k\in\OP_{r;\be,\al}$. Writing $S=B_k$, we have that $B_1/\cdots/B_{k-1}\in\OPa_{r;\be-\chi_S,\al^-}$. Since each element in the ordered partition $B_1/\cdots/B_{k-1}$ that is bigger than $\min(S)$ creates an inversion with the last block, Equation (\ref{eqI}) follows immediately.
\end{proof}

Summing over all the strong compositions $\al$ of $n$ with $k$ parts, we have the following corollary.
\begin{corollary}\label{corollary:Irecursion}
	The generating function $I_{r;\be,k}(q)$ satisfies the following equation:
	\begin{equation}\label{eqI1}
	I_{r;\be,k}(q) = \sum_{\substack{S\subseteq[m],\ \chi_S\leq \beta}}
	q^{\sum_{i=min(S)+1}^{m}(\be_i-\chi_S(i))} \Ia_{r;\be-\chi_S,k-1}(q).
	\end{equation}
\end{corollary}

We shall prove a similar result about the statistic minimaj in the following subsection.

\subsection{The recursion for minimaj}

In our new notation, \coref{Dplus} shows that
\begin{equation}\label{IMeq}
\Ia_{r;\be,k}(q) = \Ma_{r;\be,k}(q).
\end{equation}
We shall prove in this subsection that 
\begin{theorem}\label{theorem:Mrecursion}
	The generating function $M_{r;\be,k}(q)$ satisfies the following equation:
	\begin{equation}\label{eqM}
	M_{r;\be,k}(q) = \sum_{\substack{S\subseteq[m],\ \chi_S\leq \beta}}
	q^{\sum_{i=min(S)+1}^{m}(\be_i-\chi_S(i))} \Ma_{r;\be-\chi_S,k-1}(q).
	\end{equation}
\end{theorem}
Then as a consequence of \coref{Irecursion}, \tref{Mrecursion} and Equation (\ref{IMeq}), we have
\begin{theorem}\label{theorem:eq2}
	For any integers $n,r$ and composition $\be$,
	$$
	D_{r;\be,k}^{\inv}(q) = D_{r;\be,k}^{\minimaj}(q).
	$$
\end{theorem}

In order to prove \tref{Mrecursion}, we need to state some combinatorial actions and properties about the statistic minimaj. We always use the setting that for any integers $n,r$, we consider ordered multiset partitions of the form $\mu=B_1/\cdots/B_k\in\OP_{r;\be,\al}$, where $\be=(\be_1,\ldots,\be_m)\vDash n$ is a weak composition and 
$\al=(\al_1,\ldots,\al_k)\svDash n$ is a strong composition. We let $\al^-=(\al_1,\ldots,\al_{k-1})$. 

A \emph{$k$-segmented word} is a pair $(w,\al)$ such that $w=w_1\cdots w_n$ is a length $n$ word and $\al$ is a strong composition of $n$. We write such $k$-segmented word in the form of a word $w$ with dots after $w_{\al_1},w_{\al_1+\al_2},\ldots,w_{\al_1+\cdots+\al_{k-1}}$. The components of the words separated by the dots are called \emph{segments}.
For example, the 3-segmented word $(3342412,(2,3,2))$ can be written as $33\cdot 424\cdot 12$.

For an ordered multiset partition $\mu=B_1/\cdots/B_k\in\OP_{r;\be,\al}$ where $B_i =\{j_1^{(i)}<\ldots<j_{\al_i}^{(i)}\}$, 
we let $w(\mu)=\mu[1]\cdot\mu[2]\cdot\ \cdots\ \cdot\mu[k]$ denote the \emph{k-segmented word} obtained in the following way: we let the last segment $\mu[k]$ be the increasing word $j_1^{(k)}\cdots j_{\al_k}^{(k)}$. For $1\leq i\leq k-1$, assume that the $i+1^{\textnormal{st}}$ segment $\mu[i+1]$ is defined and let $r$ be the first letter of $\mu[i+1]$. Let $j_1^{(i)},\ldots,j_m^{(i)}$ be the numbers that are less than or equal to $r$, and let $j_{m+1}^{(i)},\ldots,j_{\al_i}^{(i)}$ be the numbers that are greater than $r$, then we define $\mu[i]=j_{m+1}^{(i)} \cdots j_{\al_i}^{(i)} j_1^{(i)}\cdots j_m^{(i)}$. We also refer to $w(\mu)$ as the permutation component of the segmented word without causing ambiguity. Note that $w(\mu)$ as a permutation coincides with our definition of $\miniword(\mu)$.
Thus we have the following lemma:
\begin{lemma}
	Let $\mu$ be an ordered multiset partition, then $\minimaj(\mu) = \maj(w(\mu))$.
\end{lemma}

Rhoades in \cite{Rhoades} defined an action on ordered multiset partitions $\mu$ to interchange the number of $i$ and $i+1$ in $\mu$, called the \emph{$t_i$-switch map}. Let $s_i$ be the action on a sequence that interchange its \thn{i} and $i+1{\textnormal{st}}$ component, then Rhoades proved the following theorem:
\begin{theorem}[Rhoades]\label{theorem:tswitch}
	There exists a bijective map
	$$
	t_i: \quad  \OP_{\be,k} \rightarrow \OP_{s_i\cdot \beta, k}
	$$
	such that $\minimaj(t_i(\mu))=\minimaj(\mu)$.
\end{theorem}
Recall that we can add 1 to all the entries of an ordered multiset partition in $\OPa_{r;(\be_1,\ldots,\be_n),k}$ to get an ordered multiset partition in $\OP_{(r,\be_1,\ldots,\be_n),k}$, we can naturally generalize \tref{tswitch} to the set 
$\OPa_{r;\be,k}$ that allows us to rearrange the component of $\beta$ and the number $r$:
\begin{corollary}\label{corollary:perms}
	Let $(\ga_0,\ga_1,\ldots,\ga_m)$ be any rearrangement of the sequence $(r,\be_1,\ldots,\be_m)$, then there is a minimaj-preserving bijection $\psi$ between the sets 
	$\OPa_{\ga_0;(\ga_1,\ldots,\ga_m),k}$ and $\OPa_{r;(\be_1,\ldots,\be_m),k}$.
\end{corollary}

It is obvious that for an ordered multiset partition,  the contribution of the last block to minimaj only depends on the minimum element of the last block. Thus we have the following lemma.
\begin{lemma}
	Let $B_1/\cdots/B_k$ be an ordered multiset partition. Then 
	\begin{equation}
	\minimaj(B_1/\cdots/B_k) = \minimaj(B_1/\cdots/\min(B_k)).
	\end{equation}
\end{lemma}

Rhoades in \cite{Rhoades} defined an action of the group $\Z_m = \langle c \rangle$ on $\OP_{\be,\al}$ by decrementing all the letters by 1 modulo $m$. Analogously, we define the group action of $\Z_{m+1} = \langle c \rangle$ on $\OPa_{r;\be,\al}$ by decrementing all the letters by 1 modulo $m+1$. 
Rhoades in \cite{Rhoades} proved that
\begin{lemma}[Lemma 3.4 in \cite{Rhoades}]\label{lemma:33rho}
	If the last component of $\al$ is 1, then $w(c.\mu) = c.w(\mu)$ for any $\mu\in\OP_{\be,\al}$.
\end{lemma}
Recall that there is a bijective relation between $\OPa_{r;(\be_1,\ldots,\be_m),\al}$ and $\OP_{(r,\be_1,\ldots,\be_m),\al}$.
It follows from \lref{33rho} and our new group action of $\Z_{m+1}$ that
\begin{lemma}\label{lemma:33new}
	If the last component of $\al$ is 1, then $w(c.\mu) = c.w(\mu)$ for any $\mu\in\OPa_{r;\be,\al}$.
\end{lemma}
Another property about the action $c$ is summarized in the following lemma:
\begin{lemma}\label{lemma:34new}
	For any word $w=w_1\cdots w_n$ with content $\{0^r, 1^{\be_1},\ldots,m^{\be_m}\}$ such that 
	$w_n\neq 0$, we have $\maj(c.w)=\maj(w)+r$.
\end{lemma}
\begin{proof}
	The map $c$ moves every descent occurring before a maximal contiguous run of $0$'s in $w$ to the position at the end of this run.
\end{proof}
Now we can prove the following lemma.
\begin{lemma}\label{lemma:lm36new}
	Given integers $n,r$. Let $\al=(\al_1,\ldots,\al_k)\svDash n$ be a strong composition with $\al_k=1$ and let $\be=(\be_1,\ldots,\be_m)\vDash n$ be a weak composition. We have
	\begin{equation}\label{l36eq}
	M_{r;\be,\al}(q) = \sum_{\be_i>0}q^{\be_{i+1}+\cdots+\be_m}
	\Ma_{r;(\be_{i+1},\ldots,\be_m,\be_1,\ldots,\be_{i}-1),\al^-}(q).
	\end{equation}
\end{lemma}
\begin{proof}
	We shall prove the recursion above about the generating function $M_{r;\be,\al}(q)$ where $\al_k=1$. Without loss of generality, we assume that $\be$ is a strong composition. Consider an ordered multiset partition $\mu\in\OP_{r;\be,\al}$. If the last block of $\mu$ is a singleton $\{m\}$, then clearly it does not contribute anything to $\minimaj(\mu)$. Writing $\mu=\mu'/m$, then $\minimaj(\mu)=\minimaj(\mu')$.
	
	Next consider the case when $\mu=\mu'/m-i$ end with $m-i$ for some $i\in\{1,\ldots,m-1\}$, then $\mu'\in\OPa_{r;(\be_1,\ldots,\be_{m-i}-1,\ldots,\be_m),\al^-}$. 
	It follows that we have the following consequence of \lref{33new} and \lref{34new}:
	\begin{eqnarray*}
		\minimaj(\mu'/m-i)& =& \minimaj(c^i.(c^{-i}.\mu'|m))\\
		&=& \minimaj(c^{-i}.\mu'|m)+ \be_{m-i+1}+\cdots +\be_m
	\end{eqnarray*}
	where $c^{-i}.\mu'\in\OPa_{\be_{m-i+1},(\be_{m-i+2},\ldots,\be_{m},r,\be_{1},\ldots,\be_{m-i}-1),\al^-}$, and we have
	\begin{equation}\label{eqntemp}
	M_{r;\be,\al}(q) = \sum_{\be_{m-i}>0}q^{\be_{{m-i}+1}+\cdots+\be_m}
	\Ma_{\be_{{m-i}+1};(\be_{{m-i}+2},\ldots,\be_m,r,\be_1,\ldots,\be_{{m-i}}-1),\al^-}(q).
	\end{equation}
	Equation (\ref{l36eq}) follows immediately from Equation (\ref{eqntemp}) and \coref{perms} since we can permute $r$ and the components of $\beta$.
\end{proof}

Now we are ready to prove \tref{Mrecursion}. \\
\textbf{Proof of \tref{Mrecursion}.}
Let $\mu=B_1/\cdots/B_k\in\OP_{r;\be,\al}$.
For the case when $\al=(\al_1,\ldots,\al_{k-1},1)$, we have the following recursion as a consequence of \lref{lm36new}:
\begin{eqnarray}
\sum_{\al^-} M_{r;\be,\al}(q) &=& \sum_{\al^-} \sum_{\be_i>0}q^{\be_{i+1}+\cdots+\be_m}
\Ma_{r;(\be_{i+1},\ldots,\be_m,\be_1,\ldots,\be_{i}-1),\al^-}(q)\nonumber\\
&=& \sum_{\be_i>0}\sum_{\al^-} q^{\be_{i+1}+\cdots+\be_m}
\Ma_{r;(\be_{i+1},\ldots,\be_m,\be_1,\ldots,\be_{i}-1),\al^-}(q)\nonumber\\
&=& \sum_{\be_i>0} q^{\be_{i+1}+\cdots+\be_m}
\Ma_{r;(\be_{i+1},\ldots,\be_m,\be_1,\ldots,\be_{i}-1),k-1}(q)\nonumber\\
&=& \sum_{\be_i>0} q^{\be_{i+1}+\cdots+\be_m}
\Ma_{r;(\be_1,\ldots,\be_{i}-1,\ldots,\be_m),k-1}(q).\label{resultM}
\end{eqnarray}
The first line is Equation (\ref{resultM}) summed over all compositions $\al^-\svDash (n-1)$ with $k-1$ parts; the second line interchanges the order of the two summations; the third line evaluates the inner sum over all possible $\al^-$'s; the last line is an application of \coref{perms}.

More generally, if the last block is of size $\al_k\geq 1$, then the following equation follows as a consequence of Equation (\ref{resultM}):
\begin{equation}
M_{r;\be,k}(q) = \sum_{\substack{B_k\subseteq[m],\ \chi_{B_k}\leq \beta}}
q^{\sum_{i=min(B_k)+1}^{m}(\be_i-\chi_{B_k}(i))} \Ma_{r;\be-\chi_{B_k},k-1}(q),
\end{equation}
which proves \tref{Mrecursion}. \hfill\qed

\section{Conclusion and future directions}
In this section, we give a brief summary of results in our paper and discuss directions for future work.

\subsection{The shared distribution}

In Sections 4 and 5, we proved the equi-distributivity of statistics inv, maj, dinv and minimaj on the set of extended ordered multiset partitions.

\begin{corollary}\label{corollary:newDun}
	For any integers $n,r$ and composition $\be$,
	$$
	D_{r;\be,k}^{\inv}(q) = D_{r;\be,k}^{\maj}(q) = 
	D_{r;\be,k}^{\dinv}(q) = 
	D_{r;\be,k}^{\minimaj}(q).
	$$
\end{corollary}


Given the work of D'Adderio, Iraci and Wyngaerd \cite{Michele}, we have the following,

\begin{theorem}[D'Adderio, Iraci and Wyngaerd]\label{theorem:michele}
	For any integers $n,k\geq 0$, we have the equality
	\begin{equation}
	\Rise_{r;n,k}[X;q,0]=
	\Rise_{r;n,k}[X;0,q]=
	\Delta'_{e_k}\Delta_{h_r}e_n|_{t=0} = 
	\Delta'_{e_k}\Delta_{h_r}e_n|_{q=0, t=q}. 
	\end{equation}
\end{theorem}

These results can be combined as follows.

\begin{corollary}
	For any integers $n,k\geq 0$, we have the equality
	\begin{multline}
	\Rise_{r;n,k}[X;q,0]=
	\Rise_{r;n,k}[X;0,q]=
	\Val_{r;n,k}[X;q,0]=
	\Val_{r;n,k}[X;0,q]
	\\=
	\Delta'_{e_k}\Delta_{h_r}e_n|_{t=0} = 
	\Delta'_{e_k}\Delta_{h_r}e_n|_{q=0, t=q}. 
	\end{multline}
\end{corollary}

Define the \emph{Mahonian distribution} on $\OP_{r;\be,k}$ to be the polynomial
$$
D_{r;\be,k}(q) := 
D_{r;\be,k}^{\inv}(q) = 
D_{r;\be,k}^{\maj}(q) = 
D_{r;\be,k}^{\dinv}(q) = 
D_{r;\be,k}^{\minimaj}(q)
$$
and let 
$
D^+_{r;\be,k}(q) := 
D_{r;\be,k}^{\inv+}(q) = 
D_{r;\be,k}^{\maj+}(q) = 
D_{r;\be,k}^{\dinv+}(q) = 
D_{r;\be,k}^{\minimaj+}(q)
$, 
then $D_{r;\be,k}(q)$ generalizes the \emph{Mahonian distribution on ordered multiset partitions} $D_{\be,k}(q)$ in \cite{wilson} that
\begin{equation*}
D_{\be,k}(q)=D_{0;\be,k}(q).
\end{equation*}
By either of the Equations (\ref{inv}), (\ref{maj}) and (\ref{dinv}), we have the base case that $D_{0;\emptyset,0}(q)=1$, $D_{r;\emptyset,k}(q)=0$ for $r+k>0$ and the recursion:
\begin{eqnarray*}
	D_{r;\be,k}(q) & =& \sum_{\ell = 0}^{k} q^{\binom{\be_n-k+\ell}{2}}
	{\ell \brack \be_n-k+\ell}_q {k \brack \ell}_q D_{r;\be^-,\ell}(q)\\
	&&   \qquad + q^{\binom{\be_n-k+\ell}{2}}
	{\ell \brack \be_n-k+\ell}_q {k-1 \brack \ell}_q \left(D^+_{r;\be^-,\ell}(q)-D_{r;\be^-,\ell}(q)\right)\\
	& =& \sum_{\ell = 0}^{k} q^{\binom{\be_n-k+\ell}{2}}
	{\ell \brack \be_n-k+\ell}_q 
	\left(
	{k-1 \brack \ell-1}_q D_{r;\be^-,\ell}(q)
	+
	{k-1 \brack \ell}_q D^+_{r;\be^-,\ell}(q)
	\right).
\end{eqnarray*}
Note that $D_{r;\be,k}(q)$ is a generalization of the \emph{$q$-Stirling number} $S_{n,k}(q)$ defined by
\begin{equation}
S_{n,k}(q)=S_{n-1,k-1}(q) + [k]_q S_{n-1,k}(q)
\end{equation}
as a consequence of the following equation due to the work of the second author \cite{wilson}:
\begin{equation}
D_{0;1^n,k}(q) = S_{n,k}(q).
\end{equation}

%

\subsection{Schur positivity}

By fixing the positions of the zero valleys in a particular Dyck path, one obtains an LLT polynomial \cite{LLT}. As a result, the combinatorial side of the Extended Delta Conjecture must be Schur positive, although there is no known Schur expansion of these polynomials. The original Delta Conjecture has two explicit (and not obviously equivalent) Schur expansions at $q$ or $t = 0$ via analysis of the major index statistic \cite{wilson} and the minimaj statistic \cite{crystal}. The latter work actually gives two proofs of Schur positivity at $q$ or $t = 0$, one using the theory of crystals and one using a bijection with skew Schur functions. The skew Schur function bijection is refined enough to carry over to the Extended Delta Conjecture case.

Given an extended ordered set partition $\pi \in \OP_{r;n,k}$, recall from Subsection 3.1 that $\miniword(\pi)$ is the word obtained by rearranging the parts of $\pi$ to minimize the major index of the resulting word. Given a nonnegative integer $\ell$, sets $D \in [n+r-1]$ and $I \in [k-1]$ of size $\ell$, and words $z \in \mathbb{N}^{\ell+1}$, $w \in \{0,1\}^{\ell+1}$, we let $M(D, I, z, w)$ be the set of $\pi \in \OP_{n,k}$ such that 
\begin{itemize}
\item $\miniword(\pi)$ has descents exactly at positions in $D$, 
\item the $j$th entry of $I$ gives the block containing the $j$th descent in $\miniword(\pi)$,
\item the $j$th weakly increasing run in $\miniword(\pi)$ begins with $z_j$ zeros, and 
\item $z_j - w_j$ of the $z_j$ zeros that begin the $j$th weakly increasing run in $\miniword(\pi)$ occur at the beginning of a block.
\end{itemize}
The map from Proposition 3.1 in \cite{crystal} gives a bijection from $M(D, I, z, w)$ to a certain set of skew Schur functions, all but one of which are vertical strips, where the zeros in $\pi$ get mapped to \emph{fixed} positions outside the skew shapes. We depict an example of this map in Figure \ref{fig:skew}, and refer the reader to \cite{crystal} for a detailed description.

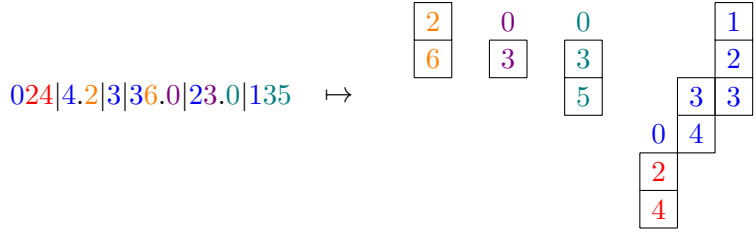
\begin{figure}
\begin{center}
\begin{tikzpicture}[scale=0.5]
\node at (0,-2.5) {${\color{blue} 0} {\color{red}24} |{\color{blue}4 }.{\color{orange} 2}|{\color{blue}3}| {\color{blue}3} {\color{orange}6}.{\color{violet} 0}|{\color{blue} 2} {\color{violet}3}. {\color{teal}0}| {\color{blue}1}{\color{teal}35}$};
\node at (5,-2.5) {$\mapsto$};
\draw (7,0) grid (8,-2);
\draw (9,-1) grid (10,-2);
\draw (11,-1) grid (12,-3);
\draw (13,-4) grid (14,-6);
\draw (14,-2) grid (15,-4);
\draw (15,0) grid (16,-3);

\node[color=orange] at (7.5,-0.5) {2};
\node[color=orange] at (7.5,-1.5) {6};
\node[color=violet] at (9.5,-0.5) {0};
\node[color=violet] at (9.5,-1.5) {3};
\node[color=teal] at (11.5,-0.5) {0};
\node[color=teal] at (11.5,-1.5) {3};
\node[color=teal] at (11.5,-2.5) {5};
\node[color=blue] at (13.5,-3.5) {0};
\node[color=red] at (13.5,-4.5) {2};
\node[color=red] at (13.5,-5.5) {4};
\node[color=blue] at (14.5,-2.5) {3};
\node[color=blue] at (14.5,-3.5) {4};
\node[color=blue] at (15.5,-0.5) {1};
\node[color=blue] at (15.5,-1.5) {2};
\node[color=blue] at (15.5,-2.5) {3};

\end{tikzpicture}
\end{center}
\caption{An example of the bijection from \cite{crystal} applied to extended ordered set partitions. Here $\pi \in \OP_{3;12,6}$ is written in minimaj order and its descents are denoted by periods. We have $D = \{4,8,11\}$, $I=\{2,4,5\}$, $z = \{1,0,1,1\}$, and $w = \{0, 0, 1, 1\}$. The leading entries in blocks and the entries in the first weakly increasing run in $\miniword(\pi)$ get sent to the ribbon shape, while other entries get sent to vertical strips.}
\label{fig:skew}
\end{figure}

Since $\minimaj$ is constant for all $\pi \in M(D, I, z, w)$ and all $\pi \in \OP_{r;n,k}$ with a fixed minimaj can be decomposed into sets of type $M(D, I, z, w)$, this proves that the distribution of minimaj over $\OP_{r;n,k}$ is a sum of products of skew Schur functions and is therefore Schur positive. By our equi-distribution results, the other three statistics also have Schur positive distributions over $\OP_{r;n,k}$. 

\begin{problem}
Provide an RSK proof \cite{wilson} or a crystal theoretic proof \cite{crystal} that the distribution of our statistics over $\OP_{n,k}$ is Schur positive.
\end{problem} 

\subsection{The Extended Delta Conjecture}

Though a number of cases of the Extended Delta Conjecture of $\Delta'_{e_k} \Delta_{h_r}e_n$ have been proved, the Extended Delta Conjecture in the general case is still open.
The main goal of this study is:
\begin{problem}
	Prove the Extended Delta Conjecture in general.
\end{problem}
\noindent This includes the original Delta Conjecture.

\tref{equi} and \coref{newDun} show that
the two versions of the Delta Conjecture and the Extended Delta Conjecture are equivalent at the case when $q$ or $t$ is 0. However, there is no proof that the combinatorial side of the two versions are equivalent in general.
\begin{problem}
	Prove that 
	\begin{equation}
	\Rise_{r;n,k}[X;q,t] = \Rise_{r;n,k}[X;t,q] = \Val_{r;n,k}[X;q,t].
	\end{equation}
\end{problem}
\noindent This includes the problem that $\Rise_{n,k}[X;q,t] = \Rise_{n,k}[X;t,q] = \Val_{n,k}[X;q,t]$.

\subsection{Other potential conjectures}

Finally, the Delta operator satisfies $\Delta_{h_r e_k}=\Delta_{s_{r,1^k}}+\Delta_{s_{r+1,1^{k-1}}}$ and $\Delta_{h_r e_k}=\Delta'_{e_k}\Delta_{h_r}+\Delta'_{e_{k-1}}\Delta_{h_r}$, so 
the Extended Delta Conjecture can be amended to involve sums of consecutive hook-shaped Schur functions in the subscript. It would be nice to have a conjecture for when just a single hook-shaped Schur function appears. 

\begin{problem}
	Give a combinatorial conjecture for the expression $\Delta_{s_\lambda} e_n$, where $\la\vdash n$ is of hook shape.
\end{problem}

\bigskip\bigskip\bigskip

\appendix

\section{Four bijections between ordered multiset partitions and parking functions}

In this appendix, we present four bijections, $\ga^{\dinv},\ga^{\maj},\ga^{\inv},\ga^{\minimaj}$, of Haglund, Remmel and the second author \cite{HRW} when they were proving  the following equations appear in \tref{combo}:
\begin{eqnarray}
\Rise_{n,k}[X;q,0]|_{M_\beta} &=& D_{\beta,k+1}^{\dinv}(q),\\
\Rise_{n,k}[X;0,q]|_{M_\beta} &=& D_{\beta,k+1}^{\maj}(q),\\
\Val_{n,k}[X;q,0]|_{M_\beta} &=& D_{\beta,k+1}^{\inv}(q),\\
\Val_{n,k}[X;0,q]|_{M_\beta} &=& D_{\beta,k+1}^{\minimaj}(q).
\end{eqnarray}
We shall omit the proof of bijectivity which can be found in \cite{HRW}.

\subsection{The bijection $\ga^{\dinv}$ of $\Rise_{n,k}[X;q,0]|_{M_\beta} = D_{\beta,k+1}^{\dinv}(q)$}

Recall that 
$$
D_{\beta,k+1}^{\dinv}(q) = \sum_{\pi\in\OP_{\beta,k+1}}q^{\dinv(\pi)},
\mbox{ and }$$
$$
\Rise_{n,k}[X;q,0]|_{M_\beta} = \sum_{
	\substack{(\PF,R)\in\WPFc_{n,n-k-1}^{\Rise},\ 
		X^{\PF} = \prod_{i=1}^{\ell(\be)}x_i^{\be_i},\ \area^-(\PF,R)=0
}}q^{\dinv(\PF)}.
$$
The map 
$$
\ga^\dinv:\OP_{\beta,k+1} \rightarrow \{(\PF,R)\in\WPFc_{n,n-k-1}^{\Rise},\ 
X^{\PF} = \prod_{i=1}^{\ell(\be)}x_i^{\be_i},\ \area^-(\PF,R)=0\}
$$
that satisfies $\dinv(\ga^\dinv(\pi))=\dinv(\pi)$ is defined as follows.

Given $\pi=\pi_1/\cdots/\pi_{k+1}\in\OP_{\beta,k+1}$ where $|\pi_i|=\al_i$, we construct a Dyck path \\
$N^{\al_{k+1}}E^{\al_{k+1}}N^{\al_{k}}E^{\al_{k}} \cdots N^{\al_{1}}E^{\al_{1}}$ which is of size $n$. Then, the rise-decorated parking function $\ga^\dinv(\pi)$ is obtained by labeling the north steps $N^{\al_{i}}$ with entries in the block $\pi_i$, and mark all the $n-k-1$ double rises. Clearly, the resulting parking function has $\area^-$ 0, and the map $\ga^\dinv$ is invertible.

For example, for an ordered multiset partition $\pi = 24/13/235$ with $\dinv(\pi)=8$, its image under the map $\ga^\dinv$ is given in \fref{eg1} which also has dinv 8.
\begin{figure}[ht]
	\centering	
	\begin{tikzpicture}[scale=0.5]
	\fillshade{1/1,2/2,3/3,4/4,5/5,6/6,7/7}
	\Dpath{0,0}{7}{7}{0,0,0,3,3,5,5,-1};
	\PFtext{0,0}{0/2,0/3,0/5,3/1,3/3,5/2,5/4};
	\fillsome{0/2/\ast,0/3/\ast,3/5/\ast,5/7/\ast}
	\end{tikzpicture}
	\caption{The image $\ga^\dinv(\pi)$ for $\pi = 24/13/235$.}
	\label{fig:eg1}
\end{figure}
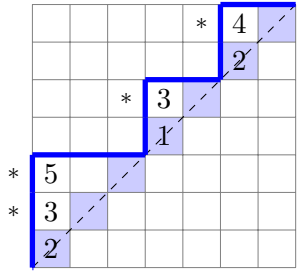

\subsection{The bijection $\ga^{\maj}$ of $\Rise_{n,k}[X;0,q]|_{M_\beta} = D_{\beta,k+1}^{\maj}(q)$}

In this section, we construct the map
$$
\ga^\maj:\OP_{\beta,k+1} \rightarrow \{(\PF,R)\in\WPFc_{n,n-k-1}^{\Rise},\ 
X^{\PF} = \prod_{i=1}^{\ell(\be)}x_i^{\be_i},\ \dinv(\PF)=0\}
$$
that satisfies $\area^-(\ga^\maj(\pi))=\maj(\pi)$.

Given $\pi=\pi_1/\cdots/\pi_{k+1}\in\OP_{\beta,k+1}$ where $|\pi_i|=\al_i$, 
we shall write the descent-starred permutation notation of $\pi$ (introduced in Section 4.2):
we first write $\pi$ as a permutation $\sg(\pi)=\sg_1\cdots\sg_n$ of the multiset $A(\be)=\{i^{\be_i}:1\leq i\leq \ell(\be)\}$ by organizing the elements in each block $\pi_i$ in decreasing order. We mark a star $\ast$ at the descent positions that are entirely contained in some block of $\pi$.
Now we are ready to construct the rise-decorated parking function $\ga^\maj(\pi)$.

We read $\sg(\pi)$ from right to left. We start with drawing a north step and labeling it with $\sg_n$ when reading the rightmost number $\sg_n$ (notice that $\sg_n$ cannot have a star mark). Inductively, suppose that the next number we read is $\sg_i$. If $\sg_i\leq\sg_{i+1}$, we add 2 steps $EN$ at the end of the previous path, and label the new north step with $\sg_i$. Otherwise when  $\sg_i>\sg_{i+1}$, we add another north step and label it with $\sg_i$ (this must be a double rise). We decorate the new north step with a star if $\sg_i$ has a star $\ast$. Then we proceed to the next number $\sg_{i-1}$.

In this way, we construct a parking function with no dinv. For example, for an ordered multiset partition $\pi = 24/13/35/2$ with $\maj(\pi)=6$, we have $\sg(\pi)=4_\ast 2 3_\ast 1 5_\ast 3 2$, and its image under the map $\ga^\maj$ is given in \fref{eg2} which has $\area^-$ 6.
\begin{figure}[ht]
	\centering	
	\begin{tikzpicture}[scale=0.5]
	\fillshade{1/1,2/2,3/3,4/4,5/5,6/6,7/7}
	\Dpath{0,0}{7}{7}{0,0,0,1,1,2,2,-1};
	\PFtext{0,0}{0/2,0/3,0/5,1/1,1/3,2/2,2/4};
	\fillsome{0/3/\ast,1/5/\ast,2/7/\ast}
	\end{tikzpicture}
	\caption{The image $\ga^\maj(\pi)$ for $\pi = 24/13/35/2$.}
	\label{fig:eg2}
\end{figure}
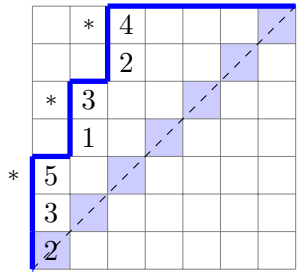

\subsection{The bijection $\ga^{\inv}$ of $\Val_{n,k}[X;q,0]|_{M_\beta} = D_{\beta,k+1}^{\inv}(q)$}

In this section, we construct the map
$$
\ga^\inv:\OP_{\beta,k+1} \rightarrow \{(\PF,V)\in\WPFc_{n,n-k-1}^{\Val},\ 
X^{\PF} = \prod_{i=1}^{\ell(\be)}x_i^{\be_i},\ \area(\PF)=0\}
$$
that satisfies $\dinv^-(\ga^\inv(\pi))=\inv(\pi)$.

Given $\pi=\pi_1/\cdots/\pi_{k+1}\in\OP_{\beta,k+1}$ where $|\pi_i|=\al_i$, we construct a diagonal $(n,n)$-Dyck path $(NE)^n$. Then we proceed from the lowest to the highest north step and from the last to the first block of $\pi$. We label the first $\al_{k+1}$ north steps increasingly with numbers in $\pi_{k+1}$, and add stars to the north steps from the second row to the $\al_{k+1}$th row. Suppose that we have completed the procedure for block $\pi_{i+1}$. For block $\pi_i$, we label the next $\al_i$ north steps increasingly with numbers in $\pi_{i}$ while adding stars to all except the first step in the $\al_i$ steps. Then we proceed to the next block $\pi_{i-1}$.

In this way, we construct a valley-decorated parking function with no area. For example, for an ordered multiset partition $\pi = 24/13/235$ with $\inv(\pi)=4$, its image under the map $\ga^\inv$ is given in \fref{eg3} which has $\dinv^-$ 4.

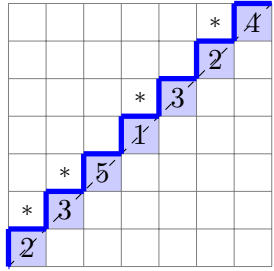
\begin{figure}[ht]
	\centering	
	\begin{tikzpicture}[scale=0.5]
	\fillshade{1/1,2/2,3/3,4/4,5/5,6/6,7/7}
	\Dpath{0,0}{7}{7}{0,1,2,3,4,5,6,-1};
	\PFtext{0,0}{0/2,1/3,2/5,3/1,4/3,5/2,6/4};
	\fillsome{1/2/\ast,2/3/\ast,4/5/\ast,6/7/\ast}
	\end{tikzpicture}
	\caption{The image $\ga^\inv(\pi)$ for $\pi = 24/13/235$.}
	\label{fig:eg3}
\end{figure}

\subsection{The bijection $\ga^{\minimaj}$ of $\Val_{n,k}[X;0,q]|_{M_\beta} = D_{\beta,k+1}^{\minimaj}(q)$}

In this section, we construct the map
$$
\ga^\minimaj:\OP_{\beta,k+1} \rightarrow \{(\PF,V)\in\WPFc_{n,n-k-1}^{\Val},\ 
X^{\PF} = \prod_{i=1}^{\ell(\be)}x_i^{\be_i},\ \dinv^-(\PF,V)=0\}
$$
that satisfies $\area(\ga^\minimaj(\pi))=\minimaj(\pi)$. $\ga^{\minimaj}$ is most technical among the four maps.

Given $\pi=\pi_1/\cdots/\pi_{k+1}\in\OP_{\beta,k+1}$ where $|\pi_i|=\al_i$, we construct $\tau=\miniword(\pi)$ as in the definition of minimaj. We define the runs of $\tau$ as its maximal, contiguous, weakly increasing subsequences. Suppose that $\tau$ has $s$ runs, then we label the runs with $0,\ldots,s-1$ from right to left. We shall construct the parking function $\ga^{\minimaj}(\pi)$ inductively by reading from the 0th to the $(s-1)$st run of $\tau$, such that the row has number in the $i$th run has area $i$ (this is sufficient for showing $\area(\ga^\minimaj(\pi))=\minimaj(\pi)$).

Suppose that $\tau_a,\tau_{a+1},\ldots,\tau_n$ is the $0$th run, and the numbers from $\tau_b$ to $\tau_n$ are contained in blocks $\pi_p,\ldots,\pi_{k+1}$ that only consist of numbers in the $0$th run (for some $b\geq a$). Suppose that the numbers $\tau_c,\ldots,\tau_{b-1}$ form the first block from right to left containing elements in run 1. 
Starting from the empty path, we first construct steps $(NE)^{n-b+1}$, filling the north steps with entries in $\pi_{k+1},\ldots,\pi_p$ increasingly for each block from bottom to top.  
We add star marks on the north steps whose labels are in the same block as the labels in the rows immediately below. 
Then we find the biggest number $\tau_d$ among $\tau_b$ to $\tau_n$ that is smaller than $\tau_c$ (which must exist by definition of miniword). We insert steps $(NE)^{a-c}$ above the north step of $\tau_d$, label the steps with  $\tau_c,\ldots,\tau_{a-1}$ from bottom to top, and add stars to the rows of $\tau_{c+1},\ldots,\tau_{a-1}$. Then we insert steps $(NE)^{b-a}$ after the east step after $(NE)^{a-c}$ that we just inserted, and label the steps with  $\tau_a,\ldots,\tau_{b-1}$ from bottom to top, adding stars to all these rows. We let $A$ denote the north step of $\tau_a$.

For greater value $i\in\{1,\ldots,s-1\}$, we suppose that the procedures for runs $0,\ldots,i-1$ have been completed and we proceed the algorithm inductively as follows. Suppose that $\tau_a,\tau_{a+1},\ldots,\tau_{n'}$ is the $i$th run that has not been read, and the numbers from $\tau_b$ to $\tau_{n'}$ are contained in blocks $\pi_p,\ldots,\pi_{k'}$ that only consist of numbers in the $i$th run (for some $b\geq a$). Suppose that the numbers $\tau_c,\ldots,\tau_{b-1}$ form the first block from right to left containing elements in run $i+1$. 

Starting from the top of $A$ in the previous procedure, we first insert steps $(NE)^{n'-b+1}$, filling the north steps with entries in $\pi_{k'},\ldots,\pi_{p}$ increasingly for each block from bottom to top. We add star marks on the north steps whose labels are in the same block as the labels in the rows immediately below.

Then we find the biggest number $\tau_d$ among $\tau_b$ to $\tau_{n'}$ that is smaller than $\tau_c$. We insert steps $(NE)^{a-c}$ above the north step of $\tau_d$, label the steps with  $\tau_c,\ldots,\tau_{a-1}$ from bottom to top, and add stars to the rows of $\tau_{c+1},\ldots,\tau_{a-1}$. Then we insert steps $(NE)^{b-a}$ after the east step after $(NE)^{a-c}$ that we just inserted, and label the steps with  $\tau_a,\ldots,\tau_{b-1}$ from bottom to top, adding stars to all these rows. We renew $A$ to be the north step of the new $\tau_a$ in this procedure. 

For example, for $\pi=13/23/14/234$, its miniword is $\tau=312341234$ which has 3 runs: $1234$, $1234$, $3$ from right to left.
The procedure of computing $\ga^\minimaj(\pi)$ is given in \fref{eg4}.

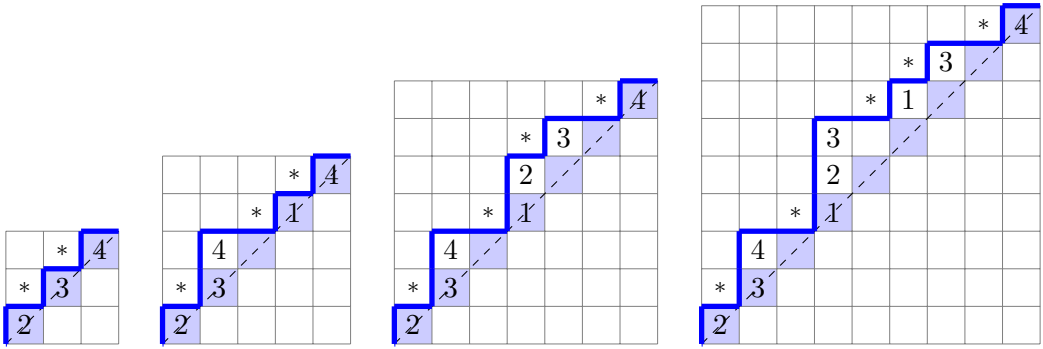
\begin{figure}[ht]
	\centering	
	\begin{tikzpicture}[scale=0.5]
	\fillshade{1/1,2/2,3/3}	
	\Dpath{0,0}{3}{3}{0,1,2,-1};
	\PFtext{0,0}{0/2,1/3,2/4};
	\fillsome{1/2/\ast,2/3/\ast}
	\end{tikzpicture}
	\quad
	\begin{tikzpicture}[scale=0.5]
	\fillshade{1/1,2/2,3/3,4/4,5/5}	
	\Dpath{0,0}{5}{5}{0,1,1,3,4,-1};
	\PFtext{0,0}{0/2,1/3,1/4,3/1,4/4};
	\fillsome{1/2/\ast,3/4/\ast,4/5/\ast}
	\end{tikzpicture}
	\quad
	\begin{tikzpicture}[scale=0.5]
	\fillshade{1/1,2/2,3/3,4/4,5/5,6/6,7/7}	
	\Dpath{0,0}{7}{7}{0,1,1,3,3,4,6,-1};
	\PFtext{0,0}{0/2,1/3,1/4,3/1,3/2,4/3,6/4};
	\fillsome{1/2/\ast,3/4/\ast,4/6/\ast,6/7/\ast}
	\end{tikzpicture}
	\quad
	\begin{tikzpicture}[scale=0.5]
	\fillshade{1/1,2/2,3/3,4/4,5/5,6/6,7/7,8/8,9/9}	
	\Dpath{0,0}{9}{9}{0,1,1,3,3,3,5,6,8,-1};
	\PFtext{0,0}{0/2,1/3,1/4,3/1,3/2,3/3,5/1,6/3,8/4};
	\fillsome{1/2/\ast,3/4/\ast,5/7/\ast,6/8/\ast,8/9/\ast}
	\end{tikzpicture}
	\caption{The procedure of computing $\ga^\minimaj(\pi)$ for $\pi = 13/23/14/234$.}
	\label{fig:eg4}
\end{figure}

\subsection{Summary}
We have presented four bijective maps of the form $\ga^{\stat}$ for $\stat=\dinv,\maj,\inv$ and $\minimaj$. In \cite{HRW},
Haglund, Remmel and the second author proved that the maps $\ga^{\stat}$ are bijective, and they map the statistic stat into some parking function statistic (stated in each section of this appendix).

Further, by checking the four bijections, we notice that each bijection $\ga^{\stat}$ maps the minimum element in the last part of $\pi$ into the car in the first row of $\ga^{\stat}(\pi)$. We use this fact to prove \tref{combor}.

\bibliographystyle{alpha}
\bibliography{val}

\end{document}